\RequirePackage{fix-cm}

\documentclass{svjour3}                     
\smartqed  
\usepackage{amsmath}  
\usepackage{amssymb}  
\usepackage{graphicx}
\usepackage{subfigure}
\usepackage{multirow}  
\usepackage[misc]{ifsym}  
\usepackage[acronym]{glossaries}  
\usepackage{color,xcolor} 
\usepackage{comment}  
\usepackage{booktabs}
\usepackage{bm} 
\usepackage[linesnumbered,ruled,vlined]{algorithm2e}

\makeglossaries 
\makeatletter
\renewcommand{\maketag@@@}[1]{\hbox{\m@th\normalsize\normalfont#1}}%
\makeatother
\printglossaries
\newacronym{TFDE}{TFDE}{time-fractional diffusion equation}

\begin{document}

\title{Solving time-fractional diffusion equations with Robin boundary conditions via fractional Hamiltonian boundary value methods}
\titlerunning{Solving TFDEs with Robin boundary conditions via FHBVMs} 

\author{Qian Luo\textsuperscript{1,2}
\and Ai-Guo Xiao\textsuperscript{1,2}
\and Xiao-Qiang Yan\textsuperscript{1,2,3*}
\and Jing-Min Xia\textsuperscript{4}
\thanks{This work is supported by NSFC (Nos. 12201537,12471391,12101525,12201636), the Research Support Funding of Xiangtan University (Nos. 22QDZ23,21QDZ16), the Youth Project of Hunan Provincial Natural Science Foundation of China (No. 2024JJ6418), the Excellent Youth Project of Hunan Provincial Department of Education (No. 24B0164), and the Postgraduate Scientific Research Innovation Project of Hunan Province in China (No. CX20240590).}}

\authorrunning{Qian Luo, Ai-Guo Xiao, Xiao-Qiang Yan, Jing-Min Xia} 

\institute{Q. Luo  \at \email{luoqian@smail.xtu.edu.cn}
\and A.G. Xiao  \at \email{xag@xtu.edu.cn}
\and X.Q. Yan  \Letter\at \email{xqyan1992@xtu.edu.cn}
\and J.M. Xia \at \email{jingmin.xia@nudt.edu.cn}
\and \textsuperscript{1} School of Mathematics and Computational Science $\&$ Hunan Key Laboratory for Computation and Simulation in Science and Engineering, Xiangtan University, Hunan 411105, China
\at \textsuperscript{2} Key Laboratory of Intelligent Computing and Information Processing of Ministry of Education, Hunan 411105, China
\at \textsuperscript{3} National Center for Applied Mathematics in Hunan, Xiangtan University, Hunan 411105, China
\at \textsuperscript{4} College of Meteorology and Oceanography, National University of Defense Technology, Changsha 410072, China}

\date{Received: date / Accepted: date}

\maketitle

\begin{abstract}
In this paper, we propose a novel numerical scheme for solving time-fractional reaction-diffusion problems with Robin boundary conditions, where the time derivative is in the Caputo sense of order $\alpha\in(0,1)$. The existence and uniqueness of the solution is proved. Our proposed method is based on the spectral collocation method in space and Fractional Hamiltonian boundary value methods in time. For the considered spectral collocation method, the basis functions used are not the standard polynomial basis functions, but rather adapt to Robin boundary conditions, and the exponential convergence property is provided. The proposed procedure achieves spectral accuracy in space and is also capable of getting spectral accuracy in time. Some numerical examples are provided to support the theoretical results.
\keywords{Time-fractional reaction-diffusion equation \and Robin boundary condition \and Spectral collocation methods \and  Fractional Hamiltonian boundary value methods}
\subclass{35R11 \and 65M70 \and 65M50 \and 65M12}
\end{abstract}

{\section{Introduction}}\label{intro}
\setcounter{theorem}{0}\setcounter{subsection}{0} \setcounter{equation}{0}\setcounter{figure}{0} \setcounter{table}{0}
Time-fractional diffusion equations have attracted considerable attention due to their wide-ranging applications in various scientific fields. They are particularly useful in describing complex phenomena in fields such as physics \cite{yuste2002subdiffusion}, chemistry and biochemistry \cite{yuste2002subdiffusion}, mechanical engineering \cite{magin2009solving}, medicine \cite{chen2010fractional}, and electronics \cite{kirane2013inverse}. Thus, there has been significant interest in developing numerical schemes for their solution.

Consider the numerical solutions to the time-fractional reaction-diffusion problem
\begin{align}\label{1.1}
_0^CD_t^{\alpha} u(x,t) = \partial_{x}^{2} u(x,t) -c(x)u(x,t) + f(x,t), \quad (x,t) \in (a,b) \times (0,T],
\end{align}
which satisfies the initial condition
\begin{align}\label{1.2}
u(x,0)=u_0(x),\quad x\in (a,b),
\end{align}
and Robin boundary conditions
\begin{align}\label{1.3}
\sigma_0u(a,t)+\beta_0u_x(a,t)=0,\quad \sigma_1u(b,t)+\beta_1u_x(b,t)=0,\quad t\in(0,T],
\end{align}
where $\alpha\in(0,1)$, $c(x)\ge 0,\ f\in \mathcal{C}([a,b]\times [0,T])$, and the initial function $u_0(x)\in \mathcal{C}([a,b])$. We assume that the coefficients $\sigma_0,\sigma_1,\beta_0$ and $\beta_1$ in the boundary conditions (\ref{1.3}) are constants and satisfy $\sigma_0>0$, $\sigma_1>0$, $\beta_0\le 0$ and $\beta_1\ge 0$. For convenience, let $G=(a,b),\ \bar{G}=[a,b],\ \Omega=(a,b)\times(0,T]$ and $\bar{\Omega}=[a,b]\times[0,T]$. In (\ref{1.1}), $_0^CD_t^{\alpha } u$ denotes the Caputo fractional derivative, which is defined as
\begin{align}\label{1.4}
_0^CD_t^{\alpha } u(x,t)=\frac{1}{\Gamma(1-\alpha)}\int_0^t(t-s)^{-\alpha}\frac{\partial u(x,s)}{\partial s} ds.
\end{align}

More general Robin boundary conditions $\sigma_0u(a,t)+\beta_0u_x(a,t)=g_1(t)$ and $\sigma_1u(b,t)+\beta_1u_x(b,t)=g_2(t)$ are easily reduced to the homogeneous case of (\ref{1.3}) by a standard linear change of variable, where $g_1(t)$ and $g_1(t)$ are smooth functions. Additionally, it is worth noting that the homogeneous Robin boundary conditions (\ref{1.3}) will become the homogeneous Dirichlet boundary conditions by setting $\beta_0=0$ and $\beta_1=0$.

A large number of numerical methods have been developed for the time-fractional reaction-diffusion problem. However, the majority of these methods primarily focus on Dirichlet boundary conditions, employing a variety of numerical techniques; see, for example, \cite{lin2007finite,mustapha2014discontinuous,alikhanov2015new,esmaeili2015pseudo,burrage2017numerical}. In addition, several researchers have also considered other boundary conditions, such as a specific class of Robin boundary conditions \cite{abu2018numerical,wang2020analysis}, Robin fractional derivative boundary conditions \cite{zhang2021efficient,niazai2022numerical} and fractional dynamic boundary conditions \cite{zhao2024direct}. To the best of our knowledge, the time-fractional diffusion problem involving a reaction term of the form $c(x)u$ has received relatively little attention. Most existing works, such as \cite{stynes2017error,huang2018optimal,gracia2018fitted}, focus on problems with Dirichlet boundary conditions. Only a few studies, like \cite{huang2019direct}, consider Robin boundary conditions. Research on such equations with Robin boundary conditions is particularly scarce, especially in the context of high-precision schemes for numerical solutions. Therefore, it is both meaningful and necessary to study the high-precision numerical schemes for problem (\ref{1.1}) subject to the Robin boundary conditions (\ref{1.3}).

Fractional Hamiltonian boundary value methods (FHBVMs) is a new class of numerical method recently introduced for solving fractional differential equations (FDEs) in \cite{brugnano2024spectrally,brugnano2024numerical,brugnano2025analysis}, as an extension of Hamiltonian boundary value methods (HBVMs). The HBVM method is a special low-rank Runge-Kutta method originally designed for Hamiltonian problems by Brugnano, et al., as discussed in the book \cite{brugnano2016line} and its references. This method has been gradually extended along several direction, such as Hamiltonian partial differential equations  \cite{barletti2018energy,brugnano2018class,luo2025hamiltonian,luo2025invariant}, delay differential equations  \cite{brugnano2022new,luo2024class}, and FDEs \cite{brugnano2024spectrally,brugnano2024numerical,brugnano2025analysis}. A main feature of HBVMs is the fact that they can gain spectrally accuracy in numerically approximate \cite{brugnano2015reprint,brugnano2014efficient}, and this feature has been recently extended to the FDE case \cite{brugnano2024spectrally,brugnano2024numerical}. Moreover, Brugnano et al. demonstrate the superiority of the FHBVM method by comparing it with other methods in \cite{brugnano2025fde}, and provide a corresponding linear stability analysis in \cite{brugnano2025analysis}. However, current research largely centers on FDEs, the application of FHBVMs in time-fractional partial differential equations has been relatively unexplored.

In this paper, we propose a novel class of numerical schemes aimed at solving time-fractional diffusion equations (\ref{1.1}) subject to general Robin boundary conditions (\ref{1.3}). We apply the line method, which first discretize the fractional partial differential equations in the spatial direction to yield a semi-discrete FDEs, and then solve the resulting FDEs by using some appropriate time algorithm. The main idea we present here is to combine a spectral collocation method for spatial discretization with FHBVMs for temporal discretization. We begin by selecting a basis function from \cite{ahmed2023highly} that is suited to the Robin boundary conditions (\ref{1.3}), and use it in the spectral collocation method to derive a semi-discrete scheme in FDEs form. This spatial discretization achieves spectral accuracy. Then, we use the FHBVM method to solve the semi-discrete FDEs in time. The time mesh is selected as a mixed mesh consisting of graded mesh near the initial time and uniform mesh thereafter, which will bring greater efficiency. It is worth mentioning that achieving spectral accuracy in space can lead to a large dimensionality of semi-discrete FDEs, resulting in significant computational costs when using general iterative methods. Therefore, to effectively avoid it, we give some setting conditions to enable the selection of fixed point iteration or blended iteration methods at grid nodes according to the situation to solve the nonlinear problems generated by FHBVMs. The proposed procedure is also capable of getting spectral accuracy in time. In addition, our scheme is superior to the one presented in \cite{lin2007finite}, that this will be demonstrated by the numerical tests in Section \ref{s5:3}.

The outline of the paper is as follows. In Section \ref{s2}, we prove the existence and uniqueness of the exact solution $u(x,t)$ to problem (\ref{1.1}). Section \ref{s3} is devoted to build a suitable spectral collocation method that adapts to Robin boundary conditions to generate the desirable semi-discrete FDEs. The time discretization FHBVM method is then applied to solve the obtained semi-discrete system, and the implementation of our method is discussed in Section \ref{s4}. Numerical examples are illustrated Section \ref{s5} to show the results of the numerical simulations with the theoretical predictions. The final section provides brief conclusions.
\vskip 0.6cm

{\section{Existence and uniqueness of the solution}\label{s2}}
\setcounter{subsection}{0} \setcounter{equation}{0}

Luchko \cite{luchko2012initial} employs the separation of variables and eigenvalue expansions to construct a classical solution to the problem (\ref{1.1}) with Dirichlet boundary conditions in an infinite series form. In this section, we use the same method to study the existence and uniqueness of the solutions to our problem (\ref{1.1}) with Robin boundary conditions (\ref{1.3}). Although the idea of the proof can be obtained roughly by imitating \cite{luchko2012initial}, we briefly give its main derivation for the sake of completeness of the paper. In this part, $\|\cdot\|_p$ is defined as the norm of the Hilbert space $L^p(G)$.

Let $\{({\lambda}_i,\psi_i):i=1,2,\cdots\}$ be the eigenvalues and eigenfunctions of the Sturm-Liouville two-point boundary problem
\begin{align}\label{2.1}
\mathcal{L}\psi_i:=-\psi''_i+c\psi_i={\lambda}_i\psi_i,\quad \sigma_0\psi_i(a)+\beta_0\psi'_i(a)=\sigma_1\psi_i(b)+\beta_1\psi'_i(b)=0,
\end{align}
defined on the interval $(a,b)$, where the eigenfunctions $\psi_i$ satisfy $\|\psi_i\|_2=1$ for all $i$. Since $\mathcal{L}$ is a positive definite and self-adjoint linear operator, it follows from the theory of eigenvalue problems that $\|\psi_i\|_\infty\le C$ \cite[p.~335]{courant1953methods} with a constant $C$ and $\lambda_i\approx (i\pi)^2/(b-a)^2>0$ \cite[p.~415]{courant1953methods} for any $i$. A standard separation of variables technique leads to the following solution (see \cite[(4.29)]{luchko2012initial})
\begin{align}\label{2.2}
u(x,t)=\sum_{i=1}^{\infty}\left[(u_0,\psi_i)E_{\alpha,1}(-{\lambda}_it^\alpha)+J_i(t)\right]\psi_i(x),
\end{align}
where $u_0$ is the initial function defined in (\ref{1.2}) and
\begin{align*}
J_i(t):=\int_0^ts^{\alpha-1}E_{\alpha,\alpha}(-{\lambda}_is^\alpha)f_i(t-s)ds,\quad f_i(t):=(f(\cdot,t),\psi_i(\cdot)),
\end{align*}
with $(\cdot,\cdot)$ representing the inner product in the space $L^2(G)$. Here, $E_{\sigma,\beta}$ is the generalised Mittag-Leffler function \cite{podlubny1998fractional}, which is defined as
\begin{align}\label{2.3}
E_{\sigma,\beta}(z):=\sum_{m=0}^{\infty}\frac{z^m}{\Gamma(\sigma m+\beta)}.
\end{align}
In particular, when $\beta=1$, $E_{\sigma,1}$ is often written as $E_\sigma$.

First, we tackle the uniqueness of the solution to the problem (\ref{1.1})-(\ref{1.3}). We begin by providing a few preliminary results, which will be needed to derive the main ones. Define the space
\begin{align*}
\mathcal{CW}^1:=\mathcal{C}(\bar{\Omega})\cap \mathcal{W}_t^1(0,T]\cap \mathcal{C}^2(G),
\end{align*}
where $\mathcal{W}_t^1(0,T]$ is the space of functions $\phi\in \mathcal{C}^1(0,T]$ such that $\phi'\in L(0,T)$, and $L(0,T)$ is the space of measurable functions that are Lebesgue integrable in $(0,T)$. The problem (\ref{1.1})-(\ref{1.3}) satisfies the maximum principle as follows.

\begin{lemma}\label{l2.1}
Let $u(x,t)\in \mathcal{CW}^1$ be a solution of (\ref{1.1}) in the domain $\bar{\Omega}$. Assume that $f(x,t) \le 0$ for $(x,t)\in \bar{\Omega}$, and that $u(x,t)$ satisfies the boundary conditions $\sigma_0u(a,t)+\beta_0u_x(a,t) \le 0$ and $\sigma_1u(b,t)+\beta_1u_x(b,t) \le 0$,\ $t\in(0,T]$. Then, either $u(x,t)\le 0,\ (x,t)\in\bar{\Omega}$ or the function $u$ attains its positive maximum on the $\bar{G}\times \{0\}$, i.e.,
\begin{align*}
u(x,t)\le \max{\left\{\max_{x\in\bar{G}}u(x,0),0\right\}},\quad \forall (x,t)\in\bar{\Omega}.
\end{align*}
\end{lemma}
\begin{proof}
We begin by assuming that the statement of the lemma does not hold, i.e., there exists a maximum point $(x^*,t^*)\in[a,b]\times(0,T]$ such that
\begin{align}\label{2.4}
u(x^*,t^*)> \max{\left\{\max_{x\in\bar{G}}u(x,0),0\right\}}>0.
\end{align}
First, consider the case where $x^*=a$. In this case, we must have $u(a,t^*)>0$. Combining this with (\ref{2.4}), we deduce that $u_x(a,t^*)\le 0$. Given that $\sigma_0>0$ and $\beta_0\le 0$, a routine computation yields $\sigma_0u(a,t^*)+\beta_0u_x(a,t^*)>0$, which contradicts our assumptions of the boundary conditions. A similar argument can be applied when $x^*=b$, leading to a contradiction in that case as well. Therefore, there can only exist a maximum point $(x^*,t^*)\in(a,b)\times(0,T]$ such that (\ref{2.4}) holds. The remaining steps of the proof are analogous to those in Theorem 2 of \cite{luchko2009maximum}, and are left to the reader.~\hfill $\square$
\end{proof}

By replacing $u$ with $-u$ in Lemma \ref{l2.1}, we can easily deduce the following lemma, which establishes the minimum principle.

\begin{lemma}\label{l2.2}
Let $u(x,t)\in \mathcal{CW}^1(G)$ be a solution of (\ref{1.1}) in the domain $\bar{\Omega}$. Assume that $f(x,t) \ge 0$ on $\bar{\Omega}$, and $u(x,t)$ satisfies the boundary conditions $\sigma_0u(a,t)+\beta_0u_x(a,t) \ge 0$ and $\sigma_1u(b,t)+\beta_1u_x(b,t) \ge 0$,\ $t\in(0,T]$. Then, either $u(x,t)\ge 0,\ (x,t)\in\bar{\Omega}$ or the function $u$ attains its negative minimum on the $\bar{G}\times \{0\}$, i.e.,
\begin{align*}
u(x,t)\ge \min{\left\{\min_{x\in\bar{G}}u(x,0),0\right\}},\quad \forall (x,t)\in\bar{\Omega}.
\end{align*}
\end{lemma}

With the help of the preceding two lemmas, we can prove the following theorem by following the approach in \cite[Theorem 4]{luchko2009maximum}, which provides an a priori estimate for the solution norm when the solution exists.
\begin{theorem}\label{t2.3}
Let $u(x,t)$ be a classical solution of the problem (\ref{1.1})-(\ref{1.3}) and $f(x,t)$ belongs to the space $\mathcal{C}(\bar{\Omega})$ with the norm $M_f :=\|f\|_{\mathcal{C}(\bar{\Omega})}$. Then the following estimate of the solution norm holds true:
\begin{align}\label{2.5}
\|u\|_{\mathcal{C}(\bar{Q})}\le M_0+\frac{T^\alpha}{\Gamma(1+\alpha)}M_f,\quad M_0:=\|u_0\|_{\mathcal{C}(\bar{G})}.
\end{align}
\end{theorem}
\begin{proof}
First, define the auxiliary function $w$ as
\begin{align*}
w(x,t):=u(x,t)-\frac{M_f}{\Gamma(1+\alpha)}t^\alpha,\quad (x,t)\in\bar{\Omega}.
\end{align*}
It is easy to verify that $w$ is the classical solution to the problem (\ref{1.1}) with the right-hand side function $\hat{f}(x,t):=f(x,t)-M_f-c(x)\frac{M_f}{\Gamma(1+\alpha)}t^\alpha$ and boundary conditions
\begin{align*}
&\sigma_0w(a,t)+\beta_0 w_x(a,t)=-\sigma_0\frac{M_f}{\Gamma(1+\alpha)}t^\alpha\le 0,\\
&\sigma_1w(b,t)+\beta_0 w_x(b,t)=-\sigma_1\frac{M_f}{\Gamma(1+\alpha)}t^\alpha\le 0.
\end{align*}
Since $c(x)\ge 0$, it follows that $\hat{f}(x,t)\le 0$ for $(x,t)\in \bar{\Omega}$. By applying Lemma \ref{l2.1}, we obtain
\begin{align*}
w(x,t)\le M_0,\quad (x,t)\in\bar{\Omega},\quad M_0:=\|u_0\|_{\mathcal{C}(\bar{G})}.
\end{align*}
Next, considering the function $u(x,t)$, it is straightforward to show that
\begin{align}\label{2.6}
u(x,t)=w(x,t)+\frac{M_f}{\Gamma(1+\alpha)}t^\alpha\le M_0+\frac{T^\alpha}{\Gamma(1+\alpha)}M_f, \quad (x,t)\in\bar{\Omega}.
\end{align}
Similarly, define another auxiliary function
\begin{align*}
\bar{w}(x,t):=u(x,t)+\frac{M_f}{\Gamma(1+\alpha)}t^\alpha,\quad (x,t)\in\bar{\Omega},
\end{align*}
and by performing the same procedure as above and using Lemma \ref{l2.2}, one obtains
\begin{align}\label{2.7}
u(x,t)\ge -M_0-\frac{T^\alpha}{\Gamma(1+\alpha)}M_f, \quad (x,t)\in\bar{\Omega}.
\end{align}
Combining (\ref{2.6}) and (\ref{2.7}), we immediately obtain the result (\ref{2.5}). The proof is completed.~\hfill $\square$
\end{proof}

Furthermore, the following theorem can be easily derived.
\begin{theorem}\label{t2.4}
The problem (\ref{1.1})-(\ref{1.3}) possesses at most one classical solution. This solution continuously depends on the data given in the problem in the sense that if $\|f-\bar{f}\|_{\mathcal{C}(\bar{\Omega})}\le \varepsilon$ and $\|u_0-\bar{u}_0\|_{\mathcal{C}(\bar{G})}\le \varepsilon_0$, then the estimate
\begin{align*}
\|u-\bar{u}\|_{\mathcal{C}(\bar{\Omega})}\le\varepsilon_0+\frac{T^\alpha}{\Gamma(1+\alpha)}\varepsilon,
\end{align*}
for the corresponding classical solutions $u$ and $\bar{u}$ holds true.
\end{theorem}
\begin{proof}
The proof only requires applying Theorem \ref{t2.3} to the function $u-\bar{u}$, which is a classical solution of the problem (\ref{1.1})-(\ref{1.3}) with the functions $f-\bar{f}$ and $u_0-\bar{u}_0$. Therefore, the proof is omitted.~\hfill $\square$
\end{proof}

\begin{remark}
Assuming that both $u$ and $\bar{u}$ are solutions to the problem (\ref{1.1})-(\ref{1.3}) with the same right-hand side $f(x,t)$ and initial condition $u_0(x)$, we obtain $\varepsilon = 0$ and $\varepsilon_0 = 0$ from Theorem \ref{t2.4}. This leads to $\|u - \bar{u}\|_{\mathcal{C}(\bar{\Omega})} = 0$, which implies that $u = \bar{u}$. Hence, the uniqueness of the solution is proved for the problem (\ref{1.1})-(\ref{1.3}).
\end{remark}

Next, define $D(\mathcal{L})$ be the space of functions $g$ that satisfy the Robin boundary conditions (\ref{1.3}) and the inclusions $g\in \mathcal{C}(\bar{\Omega})\cap \mathcal{C}^2(G)$, $\mathcal{L}(g)\in L^2(G)$ is denoted, where $\mathcal{L}$ is defined in (\ref{2.1}). The main theorem of this section is then stated as follows.

\begin{theorem}\label{t2.5}
Assume that $u_0(x)\in D(\mathcal{L}),\ \mathcal{L}(u_0)\in D(\mathcal{L})$ and $f\in D(\mathcal{L})$ for any $t\in(0,T]$. Then solution of the problem (\ref{1.1})-(\ref{1.3}) exists, is unique and is given by (\ref{2.2}) for any $(x,t)\in [a,b]\times(0,T]$.
\end{theorem}
\begin{proof}
The uniqueness of the solution has been established in the previous derivation, while the proof of existence follows a similar approach to that in \cite[Theorem 4.1]{luchko2012initial}, since $\lambda_i>0$ still holds and the proof does not depend on the boundary conditions.~\hfill $\square$
\end{proof}
\vskip 0.6cm

{\section{Space semi-discretization}\label{s3}}
\setcounter{subsection}{0} \setcounter{equation}{0}

In this section, we apply the spectral collocation method for spatial discretization. It is well known that spectral collocation method has become a widely used and effective approach for solving partial differential equations in recent years. The main idea of the method is to represent the solution as sum of truncated series by using suitable basis functions (see e.g., \cite{khater2008chebyshev,dehghan2011spectral,thirumalai2018spectral}). The chosen basis functions should be adapted to the problem itself, easy to calculate, and should converge quickly when the truncation parameter is large, and the numerical solution should have high accuracy with respect to the spatial variable. Therefore, for the problem (\ref{1.1})-(\ref{1.3}), we adopt the "Robin-Modified Chebyshev polynomials of the first kind" introduced in \cite{ahmed2023highly} as basis functions, since they are specifically constructed to satisfy the Robin boundary conditions (\ref{1.3}).
\vskip 0.5cm

{\subsection{Robin-Modified Chebyshev polynomials of first-kind}\label{s3.1}}
Initially, we will introduce the basis function considered, namely the "Robin-Modified Chebyshev polynomials of the first kind" (denoted as \textbf{RMCP1}), along with its related properties. The RMCP1, which originate from \cite{ahmed2023highly}, are defined as follows:
\begin{align}\label{3.1}
\phi_k(x)=q_k(x)T^*_k(x;a,b),\quad k=0,1,2,\cdots,
\end{align}
where $T^*_k(x;a,b)$ is defined as
\begin{align*}
T^*_k(x;a,b)=\cos\left(k \arccos\left(\frac{2x-a-b}{b-a}\right)\right),
\end{align*}
which represents the shifted orthogonality Chebyshev polynomials of the first kind, orthogonal over the interval $[a,b]$ (see, e.g., \cite{mason2002chebyshev,shen2011spectral}). Furthermore, the function $q_k(x)$ in (\ref{3.1}) has the form
\begin{align*}
q_k(x)=x^2+A_kx+B_k,
\end{align*}
in which the coefficients $A_k$ and $B_k$ are unique constants defined by
\begin{align*}
\begin{aligned}
A_k=&\frac{1}{v}(2\beta_0(\sigma_1L(a+k^2\hat{r})+2\beta_1(k^2+1)k^2\hat{r})-\sigma_0L(2\beta_1(b+k^2\hat{r})+\sigma_1L\hat{r})),\\ B_k=&\frac{1}{v}(\beta_1(a\sigma_0L(-a+2b(k^2+1))+2\beta_0(k^2+1)(L^2-2abk^2))\\
&+\sigma_1bL(\beta_0(b-2a(k^2+1))+a\sigma_0L),
\end{aligned}
\end{align*}
with the coefficients $\sigma_0,\sigma_1,\beta_0$ and $\beta_1$ correspond to the Robin boundary conditions, and $L=b-a$, $\hat{r}=b+a$ and
\begin{align*}
v=\sigma_0\sigma_1L^2-4\beta_0\beta_1(k^2+1)k^2-(2k^2+1)L(\sigma_1\beta_0-\sigma_0\beta_1)\ne 0.
\end{align*}
It must be emphasized that the basis function $\phi_k(x)$ in (\ref{3.1}) satisfies the Robin boundary conditions (\ref{1.3}). Furthermore, according to Leibniz's rule, the proposed RMCP1 will satisfy
\begin{align*}
&\phi_k(0)=B_kT^*_k(0;a,b),\\
&\phi_k^{(1)}(0)=B_kT^{*(1)}_k(0;a,b)+A_kT^*_k(0;a,b),
\end{align*}
and for $2\le q\le k+2$, we have
\begin{align*}
\phi_k^{(q)}(0)=B_kT^{*(q)}_k(0;a,b)+qA_kT^{*(q-1)}_k(0;a,b)+q(q-1)T^{*(q-2)}_k(0;a,b).
\end{align*}
From \cite[Lemma 2.1]{ahmed2023highly}, the above $T^{*(q)}_k(0;a,b)$ with $1\le q\le k+2$ is defined by
\begin{align*}
T^{*(q)}_k(0;a,b)=\frac{k(-1)^{k-q}q!(k+q-1)!(\frac{4}{b-a})^q}{(2q)!(k-q)!}{}_2F_1 \left( \left. \begin{array}{c}
q - k,\ k + q \\
q+\frac{1}{2}
\end{array} \right| \frac{a}{a - b} \right),
\end{align*}
where ${}_2F_1$ is the generalized hypergeometric function defined in \cite{koekoek2010hypergeometric}.

In the following, we will state an important theorem which show the first derivative of $\phi_n(x)$ in terms of these polynomials themselves. This result can be found in \cite{ahmed2023highly}.
\begin{theorem}\label{t3.1}
The first derivative of $\phi_n(x)$ for all $n\ge0$ can be written in the form:
\begin{align}\label{3.2}
D\phi_n(x)=\sum_{j=0}^{n-1}a_j(n)\phi_j(x)+\epsilon_n(x),\quad \epsilon_n(x)=e_1(n)x+e_0(n),
\end{align}
where the expansion coefficients $a_0(n),a_1(n),\cdots,a_{n-1}(n)$ are the solution to the system
\begin{align*}
G_na_n=b_n,
\end{align*}
with $a_n=[a_0(n),a_1(n),\cdots,a_{n-1}(n)]^T,\ G_n=(g_{ij}(n))_{0\le i,j\le n-1}$, and $b_n=[b_0(n),$ $b_1(n),\cdots,b_{n-1}(n)]^T$. The elements of $G_n$ and $b_n$ are defined as follows:
\begin{align*}
g_{ij}(n)=\left\{\begin{matrix}
 \phi_{j-1}^{(i+1)}(0), & j\ge i\\
  0,& \text{otherwise}
\end{matrix}\right.,\qquad  b_i(n)=\phi_n^{(i+3)}(0).
\end{align*}
And the two coefficients $e_0(n)$ and $e_1(n)$ in (\ref{3.2}) are given by
\begin{align*}
e_0(n)=\phi_n^{(1)}(0)-\sum_{j=0}^{n-1}a_j(n)\phi_j(0),\quad e_1(n)=\phi_n^{(2)}(0)-\sum_{j=0}^{n-1}a_j(n)\phi^{(1)}_j(0).
\end{align*}
\end{theorem}

By defining the vector
\begin{align}\label{3.3}
\Phi(x)=[\phi_0(x),\phi_1(x),\cdots,\phi_N(x)]^T,
\end{align}
it is straightforward to obtain the operational matrix of derivatives of $\Phi(x)$ from Theorem \ref{t3.1}. As a result, the following important finding, presented in \cite{ahmed2023highly}, is obtained.

\begin{corollary}\label{c3.2}
The mth derivative of the vector $\Phi(x)$ has the form:
\begin{align}\label{3.4}
\frac{d^m \Phi(x)}{dx^m}=H^m\Phi(x)+\eta^{(m)}(x),
\end{align}
where
\begin{align}\label{3.5}
\eta^{(m)}(x)=\sum_{k=0}^{m-1}H^k\epsilon^{(m-k-1)}(x),\quad \epsilon(x)=[\epsilon_0(x),\epsilon_1(x),\cdots,\epsilon_N(x)]^T,
\end{align}
and $H=(h_{ij})_{0\le i,j\le N}$ with
\begin{align}\label{3.6}
h_{ij}=\left\{\begin{matrix}
 a_j(i), & \quad i>j,\\
  0,&\quad \text{otherwise}.
\end{matrix}\right.
\end{align}
\end{corollary}
\vskip 0.5cm

{\subsection{Robin shifted Chebyshev first-kind collocation method}\label{s3.2}}

In this subsection, with the aim of using the spectral collocation method \cite{canuto2006spectral} to generate the semi-discretized system, we consider the collocation method based on the RMCP1 (\ref{3.1}) to obtain numerical solutions for the problem (\ref{1.1}) on the interval $[a,b]$. The unknown exact solution $u(x,t)$ is hence approximated by a finite series expansion
\begin{align}\label{3.7}
u_N(x,t)=\sum_{k=0}^{N}y_k(t)\phi_k(x)\approx u(x,t),
\end{align}
where $y_k (t)$ is the unknown time-dependent function for $k=0,1,\cdots,N$. For sake of brevity, by defining the vector of unknowns as $y(t)=[y_0(t),y_1(t),\cdots,y_N(t)]^T$ and combining (\ref{3.3}), the equation (\ref{3.7}) can then be rewritten in vector form as follows:
\begin{align}\label{3.8}
u_N(x,t)=\Phi^T(x) y(t).
\end{align}
Based on the facts of Corollary \ref{c3.2}, it is a simple matter to
\begin{align}\label{3.9}
\frac{d^2 \Phi(x)}{dx^2}=H^2\Phi(x)+\eta^{(2)}(x),
\end{align}
in which $\eta^{(2)}(x)$ and $H$ are given in (\ref{3.5}) and (\ref{3.6}), respectively. Then, by using the formula (\ref{3.9}), it is therefore possible to represent the spatial second-order partial derivative of the approximated solution $u_N(x,t)$ as follows
\begin{align}\label{3.10}
\notag\frac{\partial^2 u_N}{\partial x^2}&=(\Phi^T(x))^{(2)}y(t)=[H^2\Phi(x)+\eta^{(2)}(x)]^Ty(t)\\
&=\Phi^T(x)(H^2)^Ty(t)+\eta^{(2)}(x)^Ty(t).
\end{align}
Similarly, the time-fractional derivative of $u_N(x,t)$ can be obtained, that is
\begin{align}\label{3.11}
_0^CD_t^{\alpha } u_N(x,t)={\Phi^T(x)} _0^CD_t^{\alpha }y(t).
\end{align}
If we substitute (\ref{3.8}), (\ref{3.10}) and (\ref{3.11}) into the equation (\ref{1.1}), the residual function of (\ref{1.1}) will be
\begin{align}\label{3.12}
\notag R_N(x,t):=&{\Phi^T(x)} _0^CD_t^{\alpha }y(t)-\left[\Phi^T(x)(H^2)^Ty(t)\right.\\
&\left.+\eta^{(2)}(x)^Ty(t)\right]+c(x)\Phi^T(x) y(t)-f(x,t).
\end{align}
The essence of the present scheme is that we obtain its residual function and force it to go to zero at certain sets of collocation points. We choose the collocation points as the $(N+1)$ zeros of $T^*_{N+1}(x)$, which are given by
\begin{align}\label{3.13}
x_k=\frac{1}{2}\left((a+b)+(b-a)\cos\left(\frac{2k+1}{2(N+1)}\pi\right)\right),\quad k=0,1,\cdots,N.
\end{align}
Then, the residual function (\ref{3.12}) must be exact at $x$ equal to $x_k$ defined in (\ref{3.13}), that is
\begin{align*}
R_N(x_k,t)=0,\quad k=0,1,2,\cdots,N.
\end{align*}
Consequently, the semi-discretization leads to the following system of FDEs,
\begin{align}\label{3.14}
\left\{\begin{aligned}
{\Psi} _0^CD_t^{\alpha }y(t)&=\left[\Psi(H^2)^T +\Lambda -C_N\Psi\right] y(t)+F(t),\\
{\Psi} y(0)&=U_0,
\end{aligned}\right.
\end{align}
where $\Psi$, $\Lambda$ and $C_N$ are all $(N+1)\times(N+1)$ dimensional matrixs, defined as
\begin{align*}
\Psi=\left[\Phi(x_0),\Phi(x_1),\cdots,\Phi(x_{N})\right]^T&,\ \Lambda=\left[\eta^{(2)}(x_0),\eta^{(2)}(x_1),\cdots,\eta^{(2)}(x_{N})\right]^T,\\
C_N={\rm{diag}}&(c(x_0),c(x_1),\cdots,c(x_{N})),
\end{align*}
respectively, and $F(t)$ and $U_0$ are the $(N+1)$-dimensional column vectors given by
\begin{align*}
F(t)=\left[f(x_0,t),f(x_2,t),\cdots,f(x_{N},t)\right]^T,\ U_0=\left[u_0(x_0),u_0(x_1),\cdots,u_0(x_{N})\right]^T.
\end{align*}

It must be emphasized that the approximate solution $u_N(x,t)$ possesses the following property.
\begin{theorem}\label{t3.3}
The approximate solution $u_N(x,t)$ naturally satisfies Robin boundary conditions (\ref{1.3}).
\end{theorem}
\begin{proof}
The proof is straightforward. We simply need to substitute the approximate solution $u_N(x,t)$ from (\ref{3.8}) into the Robin boundary conditions (\ref{1.3}), and thus get
\begin{align*}
\sigma_0u_N(a,t)+\beta_0(u_N)_x(a,t)&=\sigma_0\Phi^T(a) y(t)+\beta_0(\Phi^T(a))^{(1)}y(t)\\
&=\left[\sigma_0\Phi^T(a)+\beta_0(\Phi^T(a))^{(1)}\right] y(t)=0,
\end{align*}
where the last equality follows from the fact that the basis function $\phi(x)$ satisfies the Robin boundary conditions (\ref{1.3}), that is
\begin{align*}
\sigma_0\Phi^T(a)+\beta_0(\Phi^T(a))^{(1)}=0.
\end{align*}
By using the same approach and the condition
\begin{align*}
\sigma_1\Phi^T(b)+\beta_1(\Phi^T(b))^{(1)}=0,
\end{align*}
it is straightforward to conclude that
\begin{align*}
\sigma_1u_N(b,t)+\beta_1(u_N)_x(b,t)=0.
\end{align*}
Thus, the proof is completed.~\hfill $\square$
\end{proof}

\begin{remark}
For each $t\in (0,T]$, and under the assumption that $\partial_x^p u(\cdot,t)\in \mathcal{C}([a,b])$ for $p=0,1,2,\dots,N$, the convergence properties of RMCP1 have been established in \cite[Theorem 6.1]{ahmed2023highly}. This result shows that the approximate solution $u_N(x,t)$ converge exponentially to the exact solution $u(x,t)$ as $N$ increases, when the spectral collocation method is applied using the basis functions defined in (\ref{3.1}) and the collocation points given in (\ref{3.13}).
\end{remark}
\vskip 0.6cm

{\section{Temporal discretization}\label{s4}}
\setcounter{subsection}{0} \setcounter{equation}{0}

In this section, we analyze the semi-discrete FDEs (\ref{3.14}) by using FHBVMs introduced in \cite{brugnano2024spectrally,brugnano2024numerical,brugnano2025analysis}. For simplicity of expression and effective application of FHBVMs, we define that
\begin{align*}
g(t,y(t))=\left[(H^2)^T+{\Psi}^{-1}\Lambda-{\Psi}^{-1}C_N\Psi\right] y(t)+{\Psi}^{-1}F(t),
\end{align*}
then the semi-discrete FDEs (\ref{3.14}) can be expressed in the following general form
\begin{align}\label{4.1}
_0^CD_t^{\alpha }y(t)=g(t,y(t)),\quad  t\in(0,T],
\end{align}
with $\alpha\in(0,1)$ and the corresponding initial condition
\begin{align}\label{4.2}
y_0:=y(0)=\Psi^{-1}U_0\in \mathbb{R}^{N+1}.
\end{align}
Here, $_0^CD_t^{\alpha }y(t)$ is defined in (\ref{1.4}) and the Riemann-Liouville intergral associated to $_0^CD_t^{\alpha }y(t)$ is given by
\begin{align}\label{4.3}
I^{\alpha }y(t)=\frac{1}{\Gamma(\alpha)}\int_0^t(t-\tau)^{\alpha-1}g(\tau,y(\tau))d\tau, \quad t\in [0,T].
\end{align}
Consequently, the solution of (\ref{4.1}) can be formally written as
\begin{align}\label{4.4}
y(t)=y_0+I^{\alpha }y(t)=y_0+\frac{1}{\Gamma(\alpha)}\int_0^t(t-\tau)^{\alpha-1}g(\tau,y(\tau))d\tau.
\end{align}

We are now in a position to briefly introduce the background information on the application of FHBVMs to the problem (\ref{4.1})-(\ref{4.2}), as detailed in \cite{brugnano2024spectrally,brugnano2024numerical,brugnano2025analysis}. The basic idea is that of deriving a local polynomial approximation of the vector field $g(t,y(t))$ in (\ref{4.1}) by expanding it along the orthonormal Jacobi polynomial basis $P_j(x)\in \Pi_j$ (see \cite[(6)]{brugnano2024spectrally}), such that (for $i,j=0,1,\cdots$)
\begin{align}\label{4.5}
P_0(x)=1,\quad \int_0^1\omega(x)P_i(x)P_j(x)dx=\delta_{ij},\quad \omega(x)=\alpha(1-x)^{\alpha-1},
\end{align}
on the interval $[0,1]$, where $\Pi_j$ is the set of polynomials of degree $j$, $\delta_{ij}$ is the Kronecker-delta symbol and $\omega(x)$ is the weighting function such that $\int_0^1\omega(x)dx=1$. For the time mesh, to cope with the possible nonsmoothness of the vector field at the initial point and avoid oscillation in subsequent solutions, we consider using a double mesh (also known as mixed mesh). This involves possible initial graded mesh
\begin{align}\label{4.6}
\begin{split}
\hat{t}_0=0,&\quad \hat{t}_i=\hat{t}_{i-1}+h_i=\hat{t}_{i-1}+r^{i-1}h_1\quad i=1,2\cdots,v,\\ &\sum_{i=1}^{v}h_i=h_1\sum_{i=1}^vr^{i-1}=h_1\frac{r^v-1}{r-1}=mh,
\end{split}
\end{align}
only in the interval $[0,mh]$, and a uniform mesh
\begin{align}\label{4.7}
t_j=jh=j\frac{T}{M},\quad j=m,m+1,\cdots,M,
\end{align}
with the timestep $h$ in the remaining interval $[mh,T]$ for $m\in\{1,2,\cdots,M\}$. It is apparent that, when $v>1$ and $m=M$, the mesh reduces to a purely graded mesh, whereas when $v=m=1$, it becomes a purely uniform mesh. For the graded mesh (\ref{4.6}), the parameter $r$ is determined by (refer to \cite{brugnano2025analysis}):
\begin{align*}
r=\left\{\begin{matrix}
 2 & \text{if}\ m=1,\\
 \frac{m}{m-1} & \text{if}\ m\ge 2,
\end{matrix}\right.
\end{align*}
and the initial timestep $h_1$ can be computed form the last equation of (\ref{4.6}), which gives
\begin{align*}
h_1=\frac{mh(r-1)}{r^v-1}.
\end{align*}
From this, we can achieve the goal of reducing $h_1$ by increasing the value of $v$.

In the following, we describe the derivation and implementation of FHBVMs on the mixed mesh defined by (\ref{4.6}) and (\ref{4.7}). First, let us denote the exact solution of (\ref{4.1}) to the sub-intervals $[\hat{t}_{i-1},\hat{t}_{i}]$ and $[t_{j-1},t_{j}]$ as
\begin{align*}
&\hat{y}_i(ch_i)\equiv y(\hat{t}_{i-1}+ch_i),\ i=1,2,\cdots,v,\\
&y_j(ch)\equiv y(t_{j-1}+ch),\ j=m+1,m+2\cdots,M,
\end{align*}
with $c\in[0,1]$, respectively. Then, by using (\ref{4.4}) and following similar steps as outlined in \cite{brugnano2024numerical,brugnano2025analysis}, for $i=1,2,\cdots,v$ we obtain
\begin{align}\label{4.8}
\notag \hat{y}_i(ch_i)=&y_0+\frac{1}{\Gamma(\alpha)}\sum_{\iota=1}^{i-1} h_{\iota}^\alpha\int_{0}^{1}\left(\frac{r^{i-\iota}-1}{r-1}+cr^{i-\iota}-\tau\right)^{\alpha-1}g(\tau h_\iota,\hat{y}_\iota(\tau h_\iota))d\tau\\
&+\frac{h_{i}^\alpha}{\Gamma(\alpha)}\int_{0}^{c}(c-\tau)^{\alpha-1}g(\tau h_{i},\hat{y}_{i}(\tau h_{i}))d\tau,
\end{align}
and for $j=m,m+1,\cdots,M$, it follows that
\begin{scriptsize}
\begin{align}\label{4.9}
\notag y_j(ch)=& y_0+\frac{1}{\Gamma(\alpha)}\sum_{\iota=1}^{v} h_{\iota}^\alpha\int_{0}^{1}\left(\frac{(j-1+c)(r^{v}-1)-m(r^{\iota-1}-1)}{mr^{\iota-1}(r-1)}-\tau\right)^{\alpha-1}g(\tau h_\iota,\hat{y}_\iota(\tau h_\iota))d\tau\\
\notag&+\frac{h^\alpha}{\Gamma(\alpha)}\sum_{\bar{\iota}=m+1}^{j-1}\int_0^1(j-\bar{\iota}+c-\tau)^{\alpha-1}g(\tau h,y_{\bar{\iota}}(\tau h))d\tau\\
&+\frac{h^\alpha}{\Gamma(\alpha)}\int_{0}^{c}(c-\tau)^{\alpha-1}g(\tau h,y_{j}(\tau h))d\tau
\end{align}
\end{scriptsize}
with $c\in[0,1]$. Next, by considering the local problem of FDEs (\ref{4.1}) and expanding the local vector fields along the orthonormal Jacobi polynomial basis (\ref{4.5}), for $c\in[0,1]$ one obtains
\begin{equation}\label{4.10}
\begin{split}
&\hat{y}^{(\alpha)}_i(ch_i)=g(ch_i,\hat{y}_i(ch_i))=\sum_{\mu\ge 0}P_{\mu}(c)\gamma_{\mu}^i(\hat{y}_i),\ i=1,2,\cdots,v,\ \hat{y}_1(0)=y_0,\\
&y^{(\alpha)}_j(ch)=g(ch,y_j(ch))=\sum_{\mu\ge 0}P_{\mu}(c)\gamma_{\mu}(y_j),\ j=m+1,m+2,\cdots,M.
\end{split}
\end{equation}
Then, the piecewise polynomial solutions $\hat{\sigma}(t),\sigma(t)\in \Pi_s$, which approximate the exact solutions of the problem (\ref{4.1}), can be obtained by truncating the vector fields in (\ref{4.10}) to finite sums with $s$ terms, that is
\begin{equation}\label{4.11}
\begin{split}
&\hat{\sigma}^{(\alpha)}_i(ch_i)=\sum_{\mu=0}^{s-1}P_{\mu}(c)\gamma_{\mu}^i(\hat{\sigma}_i),\quad i=1,2,\cdots,v,\quad \hat{\sigma}_1(0)=y_0,\\
&\sigma^{(\alpha)}_j(ch)=\sum_{\mu=0}^{s-1}P_{\mu}(c)\gamma_{\mu}(\sigma_j),\quad j=m+1,m+2,\cdots,M,\quad c\in[0,1],
\end{split}
\end{equation}
where the Fourier coefficients $\gamma_{\mu}^i(z)$ and $\gamma_{\mu}(z)$ with any suitable given function $z$ are defined by
\begin{equation}\label{4.12}
\begin{split}
&\gamma_{\mu}^i(z)=\int_0^1\omega(c)P_{\mu}(c)f(ch_i,z(ch_i))dc,\\
&\gamma_{\mu}(z)=\int_0^1\omega(c)P_{\mu}(c)f(ch,z(ch))dc,\ \mu=0,1,\cdots.
\end{split}
\end{equation}
By formally replacing the vector fields $g$ in (\ref{4.8}) and (\ref{4.9}) with the truncated ones in (\ref{4.11}), it is easy to derive the corresponding expression of the approximate solutions when $c\in[0,1]$,
\begin{small}
\begin{subequations}
\begin{align}
\notag \hat{\sigma}_i(ch_i)=&y_0+\frac{1}{\Gamma(\alpha)}\sum_{\iota=1}^{i-1} h_{\iota}^\alpha\int_{0}^{1}\left(\frac{r^{i-\iota}-1}{r-1}+cr^{i-\iota}-\tau\right)^{\alpha-1}\sum_{\mu=0}^{s-1}P_{\mu}(\tau)\gamma_{\mu}^\iota(\hat{\sigma}_\iota)d\tau\\
\notag&+\frac{h_{i}^\alpha}{\Gamma(\alpha)}\int_{0}^{c}(c-\tau)^{\alpha-1}\sum_{\mu=0}^{s-1}P_{\mu}(\tau)\gamma_{\mu}^i(\hat{\sigma}_i))d\tau\\
=:&\hat{\varphi }_i(c)+h_i^\alpha\sum_{\mu=0}^{s-1}I^\alpha P_\mu(c)\gamma_{\mu}^i(\hat{\sigma}_i),\ i=1,2,\cdots,v,\label{4.13a}\\
\notag\sigma_j(ch)=&y_0+\frac{1}{\Gamma(\alpha)}\sum_{\iota=1}^{v} h_{\iota}^\alpha\int_{0}^{1}\left(\frac{(j-1+c)(r^{v}-1)-m(r^{\iota-1}-1)}{mr^{\iota-1}(r-1)}-\tau\right)^{\alpha-1}\\
\notag&\sum_{\mu=0}^{s-1}P_{\mu}(\tau)\gamma_{\mu}^\iota(\hat{\sigma}_\iota)d\tau
+\frac{h^\alpha}{\Gamma(\alpha)}\sum_{\bar{\iota}=m+1}^{j-1}\int_0^1(j-\bar{\iota}+c-\tau)^{\alpha-1}\sum_{\mu=0}^{s-1}P_{\mu}(\tau)\gamma_{\mu}({\sigma}_{\bar{\iota}})d\tau\\
\notag&+\frac{h^\alpha}{\Gamma(\alpha)}\int_{0}^{c}(c-\tau)^{\alpha-1}\sum_{\mu=0}^{s-1}P_{\mu}(c)\gamma_{\mu}(\sigma_j)d\tau\\
=:&\varphi_j(c)+h^\alpha\sum_{\mu=0}^{s-1}I^\alpha P_\mu(c)\gamma_\mu(\sigma_j),\ j=m+1,m+2,\cdots,M,\label{4.13b}
\end{align}
\end{subequations}
\end{small}
where $\hat{\varphi }_i(c)$ and $\varphi_j(c)$ represent the memory terms as follows
\begin{align*}
\hat{\varphi }_i(c)=&y_0+\sum_{\iota=1}^{i-1}h_\iota^\alpha\sum_{\mu=0}^{s-1}J_\mu^\alpha\left(\frac{r^{i-\iota}-1}{r-1}+cr^{i-\iota}\right)\gamma_\mu^\iota(\hat{\sigma}_\iota),\\
\varphi_j(c)=&y_0+\sum_{\iota=1}^{v} h_{\iota}^\alpha\sum_{\mu=0}^{s-1}J_\mu^\alpha\left(\frac{(j-1+c)(r^{v}-1)-m(r^{\iota-1}-1)}{mr^{\iota-1}(r-1)}\right)\gamma_\mu^\iota(\hat{\sigma}_\iota)\\
&+h^\alpha\sum_{\bar{\iota}=m+1}^{j-1}\sum_{\mu=0}^{s-1}J_\mu^\alpha(j-\bar{\iota}+c)\gamma_\mu(\sigma_{\bar{\iota}})
\end{align*}
with the function $J_\mu^\alpha(x)$ defined by
\begin{align}\label{4.14}
J_\mu^\alpha(x)=\frac{1}{\Gamma(\alpha)}\int_0^1(x-\tau)^{\alpha-1}P_\mu(\tau)d\tau,
\end{align}
and $I^\alpha P_\mu(c)$ is the Riemann-Liouville integral of $P_\mu(c)$, as defined in (\ref{4.3}).

It must be emphasized that the approximate solutions (\ref{4.13a}) and (\ref{4.13b}) involve integral terms that cannot be directly calculated. Therefore, to obtain a practical numerical method, the Fourier coefficients $\gamma_{\mu}^i(\hat{\sigma}_i)$ and $\gamma_{\mu}(\sigma_j)$ in (\ref{4.12}) need to be approximated numerically. Here, we consider using Gauss-Jacobi quadrature of order $2k$ for a suitable $k\ge s$, with abscissae $c_{\bar{i}}$ given by $P_k(c_{\bar{i}})=0,\bar{i}=1,2,\cdots,k$, and the corresponding weight $b_{\bar{i}}$ for $\bar{i}=1,2,\cdots,k$, yielding the approximations for $\mu=0,1,\cdots,s-1$,
\begin{align}\label{4.15}
\gamma_{\mu}^i(\hat{\sigma}_i)\approx \sum_{{\bar{i}}=1}^kb_{\bar{i}}P_\mu(c_{\bar{i}})g(c_{\bar{i}}h_i,\hat{\sigma}_i(c_{\bar{i}}h_i))=:\hat{\gamma}_{\mu}^i,\quad i=1,2,\cdots,v,
\end{align}
and
\begin{align}\label{4.16}
\gamma_{\mu}(\sigma_j)\approx \sum_{{\bar{i}}=1}^kb_{\bar{i}}P_\mu(c_{\bar{i}})g(c_{\bar{i}}h,\sigma_j(c_{\bar{i}}h))=:\hat{\gamma}_{\mu}^j,\quad j=m+1,m+2,\cdots,M.
\end{align}
Meanwhile, we also need to numerically evaluate the following integrals involved in (\ref{4.13a}) and (\ref{4.13b}) for $\mu=0,1,\cdots,s-1$ and $\bar{i}=1,2,\cdots,k$:
\begin{subequations}
\begin{align}
&I^\alpha P_{\mu}(c_{\bar{i}}),\label{4.17a}\\
&J_\mu^\alpha(j+c_{\bar{i}}),\quad j=1,2,\cdots,M-m-1,\label{4.17b}\\
&J_\mu^\alpha\left(\frac{r^i-1}{r-1}+c_{\bar{i}}r^i\right),\quad i=1,2,\cdots,v-1,\label{4.17c}\\
&J_\mu^\alpha\left(\frac{(j+c_{\bar{i}})(r^v-1)-m(r^i-1)}{mr^i(r-1)}\right),i=0,1,\cdots,v-1,j=m,m+1,\cdots,M-1,\label{4.17d}
\end{align}
\end{subequations}
where the expression of $J_\mu^\alpha(x)$ is given by (\ref{4.14}). From the computational point of view, the integrals (\ref{4.17a})-(\ref{4.17d}) can be pre-computed once for all, allowing for later use. In order to calculate these integrals (\ref{4.17a}), (\ref{4.17b}) and (\ref{4.17c}) on the mixed mesh defined by (\ref{4.6}) and (\ref{4.7}) effectively and efficiently, we adopt the method described in Section 3.2 of \cite{brugnano2024numerical}. In addition, the integral (\ref{4.17d}) also can be computed following the approximation procedure outlined in \cite{brugnano2024numerical} for computing the integrals (\ref{4.17b}) and (\ref{4.17c}).

Based on the above facts, the fully discrete approximations corresponding to (\ref{4.13a}) and (\ref{4.13b}) will be obtained. For sake of brevity, we shall continue using $\hat{\sigma}_i$ and $\sigma_j$ to represent the fully discrete approximations, and we obtain the following expressions
\begin{align}\label{4.18}
\begin{split}
&\hat{\sigma}_i(ch_i)=\hat{\varphi }_i(c)+h_i^\alpha\sum_{\mu=0}^{s-1}I^\alpha P_\mu(c)\hat{\gamma}_{\mu}^i,\ i=1,2,\cdots,v\\
&\sigma_j(ch)=\varphi_j(c)+h^\alpha\sum_{\mu=0}^{s-1}I^\alpha P_\mu(c)\hat{\gamma}_\mu^j,\ j=m+1,m+2,\cdots,M.
\end{split}
\end{align}
As is clear, the coefficients $\hat{\gamma}_{\mu}^l,\ l=1,2,\cdots,i-1$ and $\hat{\gamma}_\mu^\rho ,\ \rho=m+1,m+2,\cdots,j-1$ for $\mu=0,1,\cdots,s-1$, used to calculate the memory terms $\hat{\varphi }_i(c)$ and $\varphi_j(c)$ respectively, have been already computed at the previous time-steps. Moreover, we find that in order to calculate the Fourier coefficients (\ref{4.15}) and (\ref{4.16}), it is sufficient to calculate (\ref{4.18}) only at the quadrature abscissae $c_1,c_2,\cdots,c_k$. By doing so, one has the discrete problems in the following form for $\mu=0,1,\cdots,s-1$,
\begin{align}\label{4.19}
\hat{\gamma}_{\mu}^i=\sum_{{\bar{i}}=1}^kb_{\bar{i}}P_\mu(c_{\bar{i}})g\left(c_{\bar{i}}h_i,\hat{\varphi }_i(c_{\bar{i}})+h_i^\alpha\sum_{l=0}^{s-1}I^\alpha P_l(c_{\bar{i}})\hat{\gamma}_{l}^i\right),\ i=1,2,\cdots,v,
\end{align}
and
\begin{align}\label{4.20}
\hat{\gamma}_{\mu}^j=\sum_{{\bar{i}}=1}^kb_{\bar{i}}P_\mu(c_{\bar{i}})g\left(c_{\bar{i}}h,\varphi _i(c_{\bar{i}})+h^\alpha\sum_{l=0}^{s-1}I^\alpha P_l(c_{\bar{i}})\hat{\gamma}_{l}^j\right),\ j=m+1,m+2,\cdots,M.
\end{align}
Once (\ref{4.19}) and (\ref{4.20}) have been solved, and based on the fact
\begin{align*}
I^\alpha P_\mu(1)=J_\mu^\alpha(1)=\frac{\delta_{\mu0}}{\Gamma(\alpha+1)},
\end{align*}
which is obtained by (\ref{4.5}) and considering $c=1$, the approximation of $y(\hat{t}_i)$ with $i=1,2,\cdots,v$ is given by
\begin{align}\label{4.21}
\hat{y}_i:=\hat{\sigma}_i(h_i)=\hat{\varphi }_i(1)+\frac{h_i^\alpha}{\Gamma(\alpha+1)}\hat{\gamma}_0^i.
\end{align}
Similarly, it is easy to get the approximation of $y(t_j)$ for $j=m+1,m+2,\cdots,M$, that is
\begin{align}\label{4.22}
{y}_j:=\sigma_j(h)=\varphi_j(1)+\frac{h^\alpha}{\Gamma(\alpha+1)}\hat{\gamma}_0^j.
\end{align}
Strictly speaking, the method (\ref{4.18}) and (\ref{4.21})-(\ref{4.22}) defines a \textbf{FHBVM}$\bm{(k,s)}$ method with $k\ge s$. Furthermore, for sake of brevity, by introducing $\bm{c}=(c_1,c_2,$ $\cdots,c_k)^T\in \mathbb{R}^k$ and the vectors
\begin{align*}
&\bm{\hat{\gamma}}^i=\begin{pmatrix}
 \hat{\gamma}_0^i\\
 \vdots\\
 \hat{\gamma}_{s-1}^i
\end{pmatrix}\in\mathbb{R}^{s(N+1)},\quad \bm{\hat{\gamma}}^j=\begin{pmatrix}
 \hat{\gamma}_0^j\\
 \vdots\\
 \hat{\gamma}_{s-1}^j
\end{pmatrix}\in\mathbb{R}^{s(N+1)},\\
&\bm{\hat{\varphi}}_i=\begin{pmatrix}
 \hat{\varphi}_i(c_1)\\
 \vdots\\
 \hat{\varphi}_i(c_k)
\end{pmatrix}\in\mathbb{R}^{k(N+1)},\quad
\bm{\varphi}_j=\begin{pmatrix}
 \varphi_j(c_1)\\
 \vdots\\
 \varphi_j(c_k)
\end{pmatrix}\in\mathbb{R}^{k(N+1)},
\end{align*}
and the matrices $\Omega=diag(b_1,b_2,\cdots,b_k)\in \mathbb{R}^{k\times k}$ and
\begin{small}
\begin{align*}
\mathcal{P}_s=\begin{pmatrix}
 P_0(c_1) & \cdots & P_{s-1}(c_1)\\
 \vdots  &   & \vdots\\
 P_0(c_k) & \cdots & P_{s-1}(c_k)
\end{pmatrix}\in\mathbb{R}^{k\times s},\quad \mathcal{I}_s^\alpha=\begin{pmatrix}
 I^\alpha P_0(c_1) & \cdots & I^\alpha P_{s-1}(c_1)\\
 \vdots  &   & \vdots\\
 I^\alpha P_0(c_k) & \cdots & I^\alpha P_{s-1}(c_k)
\end{pmatrix}\in\mathbb{R}^{k\times s},
\end{align*}
\end{small}
the discrete problem (\ref{4.19}) can be cast in the following vector form
\begin{align}\label{4.23}
\bm{\hat{\gamma}}^i=\mathcal{P}_s^T\Omega\otimes I_{N+1} g\left(\bm{c}h_i,\bm{\hat{\varphi}}_i+h_i^\alpha\mathcal{I}_s^\alpha\otimes I_{N+1}\bm{\hat{\gamma}}^i\right),\ i=1,2,\cdots,v,
\end{align}
and (\ref{4.20}) becomes
\begin{small}
\begin{align}\label{4.24}
\bm{\hat{\gamma}}^j=\mathcal{P}_s^T\Omega\otimes I_{N+1} g\left(\bm{c}h,\bm{\varphi}_i+h^\alpha\mathcal{I}_s^\alpha\otimes I_{N+1} \bm{\hat{\gamma}}^j\right),j=m+1,m+2,\cdots,M,
\end{align}
\end{small}
where $I_{N+1}$ is the $(N+1)$-dimensional identity matrix.

\begin{remark}
In fact, the FHBVM methods can be viewed as a class of Runge-Kutta methods. By defining the $k$-dimensional block vectors
\begin{align}\label{4.25}
\hat{Y}^i:=\bm{\hat{\varphi}}_i+h_i^\alpha\mathcal{I}_s^\alpha\otimes I_{N+1}\bm{\hat{\gamma}}^i,\ i=1,2,\cdots,v,
\end{align}
and
\begin{align}\label{4.26}
Y^j:=\bm{\varphi}_i+h^\alpha\mathcal{I}_s^\alpha\otimes I_{N+1} \bm{\hat{\gamma}}^j,\ j=m+1,m+2,\cdots,M,
\end{align}
these vectors are used as arguments in the function $g$ on the right-hand side of equations (\ref{4.23}) and (\ref{4.24}), respectively.  Then, $\hat{Y}^i$ and $Y^j$ can be regarded as the stage vectors of a Runge-Kutta method, tailored to the problem at hand. Substituting (\ref{4.23}) into (\ref{4.25}) and using $\hat{Y}^i$ as the argument of $g$ yields
\begin{align}\label{4.27}
\hat{Y}^i = \bm{\hat{\varphi}}_i + h_i^\alpha \mathcal{I}_s^\alpha \mathcal{P}_s^T \Omega \otimes I_{N+1} g(\bm{c}h_i,\hat{Y}^i).
\end{align}
Similarly, substituting (\ref{4.24}) into (\ref{4.26}) gives
\begin{align}\label{4.28}
Y^j = \bm{\varphi}_i + h^\alpha \mathcal{I}_s^\alpha \mathcal{P}_s^T \Omega \otimes I_{N+1} g(\bm{c}h,Y^j).
\end{align}
From (\ref{4.27}) and (\ref{4.28}), the results suggest that the FHBVM method is a $k$-stage Runge-Kutta method with the Butcher matrix $\mathcal{I}_s^\alpha\mathcal{P}_s^T\Omega$. It is worth noting that equations (\ref{4.27}) and (\ref{4.28}) are equivalent to equations (\ref{4.23}) and (\ref{4.24}), respectively. However, the efficiency of solving (\ref{4.23}) and (\ref{4.24}) is higher than that of solving (\ref{4.27}) and (\ref{4.28}), because the discrete system corresponding to the former has a block dimension $s$ that is independent of $k$, whereas the discrete system in the latter has a block dimension of $k$. In practical applications, $s$ is usually smaller than $k$.
\end{remark}

Finally, for completeness sake, the implementation is briefly described below. We consider solving the nonlinear problems (\ref{4.23}) and (\ref{4.24}), for which two iterative methods are considered: fixed-point iteration and blended iteration. To streamline our discussion, we only provide the detailed process of solving (\ref{4.23}) here, and the process of solving (\ref{4.24}) is similar to that of solving (\ref{4.23}). For the fixed-point iteration, the algorithm applied to the nonlinear equation (\ref{4.23}) is as follows
\begin{align}\label{4.29}
\bm{\hat{\gamma}}^{i,l}=\mathcal{P}_s^T\Omega\otimes I_{N+1} g\left(\bm{c}h_i,\bm{\hat{\varphi}}_i+h_i^\alpha\mathcal{I}_s^\alpha\otimes I_{N+1}\bm{\hat{\gamma}}^{i,l-1}\right),\quad l=1,2,\cdots,
\end{align}
which can be conveniently started from $\bm{\hat{\gamma}}^{i,0}=\bm{0}$, and the following theorem holds.

\begin{theorem}\label{t4.1}
Assume $g$ be Lipchitz with constant $\bar{L}$ in the interval $[t_{i-1},t_i],\ i=1,2,\cdots,v$. Then, the iteration (\ref{4.29}) is convergent for all timesteps $h_i$ such that
\begin{align*}
h_i^\alpha \bar{L}\|\mathcal{P}_s^T\Omega\|\|\mathcal{I}_s^\alpha\|< 1.
\end{align*}
\end{theorem}
\begin{proof}
For a rigorous proof of this theorem the reader is referred to \cite[Theorem 2]{brugnano2024spectrally}.
\end{proof}

Furthermore, by defining
\begin{align*}
X_s^\alpha:=\mathcal{P}_s^T\Omega \mathcal{I}_s^\alpha,
\end{align*}
we introduce another type of iteration, namely the blended iteration, which is of Newton-type. A detailed analysis of its application to FDEs can be found in \cite{brugnano2024numerical}. Let us introduce the following equivalent form of the nonlinear equation (\ref{4.23}), that is
\begin{align*}
G(\bm{{\gamma}}^{i,l}):=\bm{\hat{\gamma}}^{i,l}-\mathcal{P}_s^T\Omega\otimes I_{N+1} g\left(\bm{c}h_i,\bm{\hat{\varphi}}_i+h_i^\alpha\mathcal{I}_s^\alpha\otimes I_{N+1}\bm{\hat{\gamma}}^{i,l}\right).
\end{align*}
The algorithm for the blended iteration is then provided as follows (refer to \cite[(61)]{brugnano2024numerical}):
\begin{algorithm}
\SetAlgoNoLine
\caption{Blended iterative methods for nonlinear system (\ref{4.23}).}
\For {$i=1,2,\ldots,v$}
{Set:~$\bm{{\gamma}}^{i,0}=\bm{0}$ and the factor\ matrix $\hat{\Psi}:=I_{N+1}-h_i^\alpha\rho_sg_0'$;}
\For {$r=0,1,\cdots,$}
{
\begin{align*}
\text{compute}:& \quad \bm{\Theta}_1^r:=-G(\bm{{\gamma}}^{i,r}),\quad \bm{\Theta}_2^r:=\rho_s(X_s^\alpha)^{-1}\otimes I_{N+1}\bm{\Theta}_1^r;\\
\text{solve}:&\quad (I_s\otimes\hat{\Psi})\bm{u}^r=\bm{\Theta}_1^r-\bm{\Theta}_2^r;\\
\text{solve}:&\quad (I_s\otimes\hat{\Psi})\Delta^r=\bm{\Theta}_2^r+\bm{u}^r;\\
\text{set}:&\quad \bm{{\gamma}}^{i,r+1}=\bm{{\gamma}}^{i,r}+\Delta^r.
\end{align*}
}
\end{algorithm}

\noindent The parameter $\rho_s$ in Algorithm 1 is defined as
\begin{align*}
\rho_s=\rm{argmin}_{\mu\in\sigma(X_s^\alpha)}\max_{\lambda\in\sigma(X_s^\alpha)}\frac{|\lambda-|\mu||^2}{2|\mu||\lambda|},
\end{align*}
in which $\sigma(X_s^\alpha)$ denotes the spectrum of $X_s$, and $g_0'$ is the Jacobian of $g$ defined in (\ref{4.1}) evaluated at the first entry of $\hat{\varphi}_i$. Form the expression of the function $g$, it becomes straightforward to obtain that
\begin{align*}
g_0'=(H^2)^T+{\Psi}^{-1}\Lambda-{\Psi}^{-1}C_N\Psi.
\end{align*}
To improve computational efficiency, this paper adopts an automatic switching strategy between the fixed-point iteration (\ref{4.29}) and the blended iteration (see Algorithm 1), based on the following condition (refer to Theorem \ref{t4.1}):
\begin{align*}
h_i^\alpha\|g_0'\|\|\mathcal{P}_s^T\Omega\|\|\mathcal{I}_s^\alpha\|\le tol,
\end{align*}
where $tol<1$ is a suitable tolerance threshold.

\begin{remark}
For the error analysis of the FHBVM($k,s$) method, the relevant theoretical analysis and numerical verification have been thoroughly discussed in \cite{brugnano2024spectrally,brugnano2025analysis}. The results demonstrate that FHBVMs achieve a spectral accuracy in time when sufficiently large value of $s$ is chosen.
\end{remark}

\begin{remark}
In order to achieve the anticipated results, we choose to use FHBVM$(k,$ $s)$ method with $k=s=22$ in the subsequent numerical tests. The motivation behind this choice is twofold. On the one hand, it is based on the linear stability analysis of the FHBVM$(22,22)$ method, as presented in \cite{brugnano2025analysis}, which demonstrates the effectiveness of high-order methods in approximating the one-parameter Mittag-Leffler function. On the other hand, this method also possesses the capability to achieve high accuracy in the time direction. It is also worth noting that the code corresponding to the FHBVM($22,22$) method, named \textbf{fhbvm2}, is an enhanced version of the original \textbf{fhbvm} code described in \cite{brugnano2024numerical}. Both \textbf{fhbvm} and \textbf{fhbvm2} are publicly available at the URL provided in \cite{fhbvm2}.
\end{remark}
\vskip 0.6cm

{\section{Numerical tests}\label{s5}}
\setcounter{subsection}{0} \setcounter{equation}{0}\setcounter{figure}{0}\setcounter{table}{0}
In this section, we will introduce some numerical experiments to illustrate and test the behavior of the methods described in the previous sections. For a detailed analysis, we report the errors computed in the $L^\infty$ norm
\begin{equation*}
e_\infty=\|u(x_k,t_i)-u_N(x_k,t_i)\|_\infty=\max_{0\le k\le N}\max_{0\le i\le M+v-m}|u(x_k,t_i)-u_N(x_k,t_i)|,
\end{equation*}
and the $L^2$ norm
\begin{align*}
e_2&=\|u(x,t)-u_N(x,t)\|_2=\left(\int_0^T\int_a^b|u(x,t)-u_N(x,t)|^2dxdt\right)^{1/2},
\end{align*}
where $u_N(x,t)$, defined in (\ref{3.7}), is the computed approximation solution, and $u(x,t)$ is the corresponding exact solution of the considered problem. All algorithms are implemented in Matlab 2019b.
\vskip 0.6cm

{\subsection{Example 1}\label{s5:1}}
Our first numerical experiment is the following time-fractional diffusion problem:
\begin{small}
\begin{align}\label{5.1}
\begin{cases}
D_t^{\alpha} u(x,t) = \partial_{x}^{2} u(x,t)-(2x+1)u(x,t)+f(x,t),\quad (x,t)\in(0,1)\times(0,1],\\
u(0,t)-2u_x(0,t)=0,\ u(1,t)+2u_x(1,t)=0,& \\
u_0(x)=0,&
\end{cases}
\end{align}
\end{small}
where $c(x)=2x+1\in(1,3)$, as assumed in the introduction. The exact analytical solution is given by
\begin{align}\label{5.2}
u(x,t)=x^2(1-x)^2e^xt^{\alpha+2},
\end{align}
then the function $f(x,t)$ can be calculated as
\begin{small}
\begin{align*}
f(x,t)=0.5e^xt^2\left[x^2(1-x)^2\Gamma (\alpha+3)-2(x^4+6x^3-8x+2-(2x+1)x^2(1-x)^2)t^\alpha\right].
\end{align*}
\end{small}
It is evident that this solution (\ref{5.2}) and the corresponding vector field $g(t,y)$ (defined in (\ref{4.1})) are smooth at the initial time $t=0$ for all $\alpha\in(0,1)$ .

We first investigate the spatial errors in two discrete norms: $L^\infty$ and $L^2$. The time direction mesh is chosen as a mixed mesh with parameters $M=200$, $v=20$, and $m=1$ to ensure that the mesh is fine enough so that the errors from the time approximation are negligible. Fig. \ref{fig1} shows the spatial errors for different truncation coefficients $N=1,3,5,7,9,10,11,12,13,14$ at different values of $\alpha=0.1$, $\alpha=0.5$, and $\alpha=0.9$. As expected from the theoretical findings, the errors exhibit an exponential decay until they reach the machine precision, ultimately stabilizing around $10^{-12}$. Next, for the temporal errors, only a uniform mesh needs to be used, as the exact solution and the corresponding vector field are smooth at $t=0$. To numerically test the accuracy of the time discretization, it can be observed from Fig. \ref{fig1} that selecting $N=10$ allows the spatial errors to achieve spectral accuracy. As a result, since the spatial error of the spectral method becomes negligible, the global error is essentially determined by the time discretization. Furthermore, Table \ref{Tab1} displays the temporal errors for two discrete norms on a uniform mesh with $h=1/M \ (M=2,3,4,5,6)$, fixed at $N=10$ and for four different values of $\alpha$. From Table \ref{Tab1}, it is evident that the temporal errors are basically of the order of the machine precision, independently of the timestep. Finally, the three-dimensional plots of the numerical solution, exact solution, and absolute error are shown in Fig. \ref{fig2}, with results that are consistent with those in Table \ref{Tab1}.
\vskip -0.7cm

\begin{figure}[ht] \centering
\includegraphics[width=5.5cm]{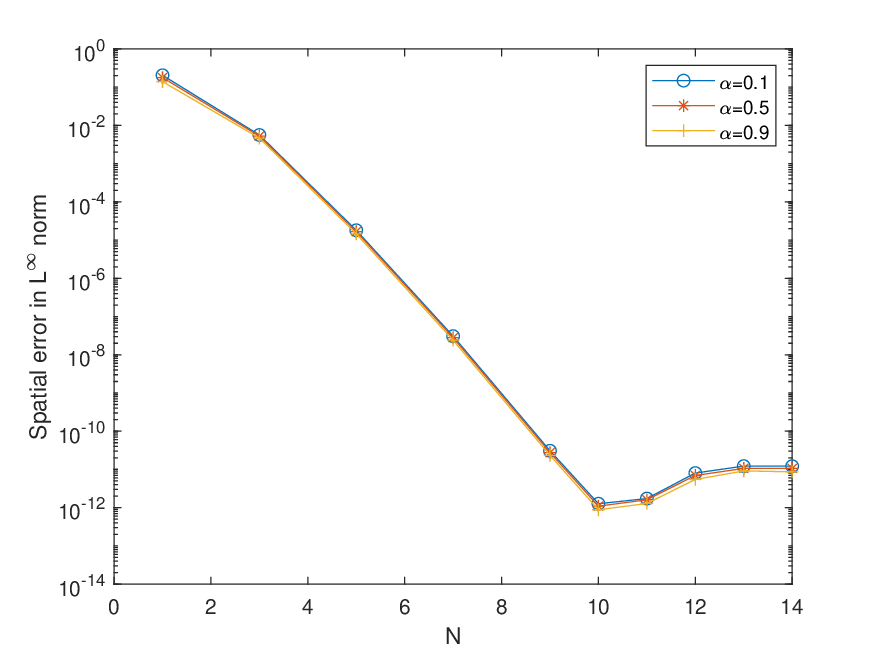}
\includegraphics[width=5.5cm]{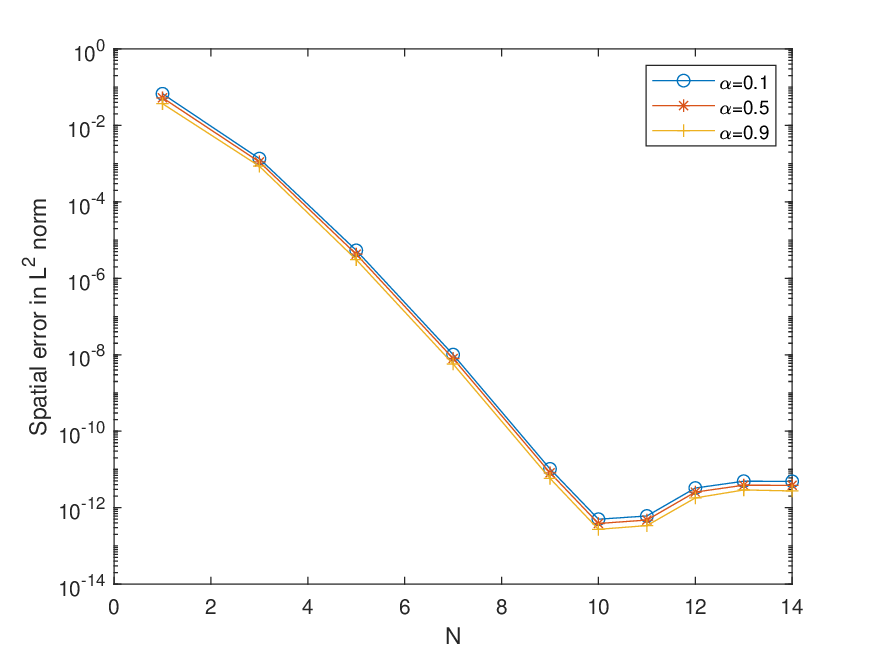}
\caption{The spatial errors in the $L^\infty$ norm (left plot) and $L^2$ norm (right plot) on the mixed mesh with $h_0=4.77\cdot10^{-9},r=2,m=1,h=0.005$ for different $\alpha$ (Example 1).}
\label{fig1}
\end{figure}

\begin{table}[htbp]
\renewcommand{\arraystretch}{1.3}
\centering
\caption{The temporal errors on uniform mesh with $h=1/M$ for four values of $\alpha$ and fixed $N=10$ (Example 1).}
\scalebox{0.86}{
\begin{tabular}{ccccccccc}
\toprule
\multirow{2}{*}{M} & \multicolumn{2}{c}{$\alpha=0.1$} & \multicolumn{2}{c}{$\alpha=0.3$} & \multicolumn{2}{c}{$\alpha=0.6$} & \multicolumn{2}{c}{$\alpha=0.9$} \\
  & $e_\infty$ & $e_2$ & $e_\infty$ & $e_2$ & $e_\infty$ & $e_2$ & $e_\infty$ & $e_2$ \\
\midrule
2 & 1.254\text{e-}12 & 5.899\text{e-}13 & 1.170\text{e-}12 & 5.430\text{e-}13 & 1.055\text{e-}12 & 4.726\text{e-}13 & 9.788\text{e-}13 & 3.856\text{e-}13 \\
3 & 1.257\text{e-}12 & 5.412\text{e-}13 & 1.168\text{e-}12 & 4.863\text{e-}13 & 1.056\text{e-}12 & 4.146\text{e-}13 & 8.949\text{e-}13 & 3.299\text{e-}13 \\
4 & 1.260\text{e-}12 & 5.221\text{e-}13 & 1.169\text{e-}12 & 4.675\text{e-}13 & 1.051\text{e-}12 & 3.919\text{e-}13 & 8.500\text{e-}13 & 3.057\text{e-}13 \\
5 & 1.261\text{e-}12 & 5.131\text{e-}13 & 1.170\text{e-}12 & 4.555\text{e-}13 & 1.047\text{e-}12 & 3.798\text{e-}13 & 8.520\text{e-}13 & 2.945\text{e-}13 \\
6 & 1.262\text{e-}12 & 5.082\text{e-}13 & 1.172\text{e-}12 & 4.515\text{e-}13 & 1.043\text{e-}12 & 3.731\text{e-}13 & 8.810\text{e-}13 & 2.873\text{e-}13 \\
\bottomrule
\end{tabular}}
\label{Tab1}
\end{table}

\begin{figure}[htbp]
\subfigure[Numerical solution] {
 \label{fig2:a}
\includegraphics[width=0.3\linewidth]{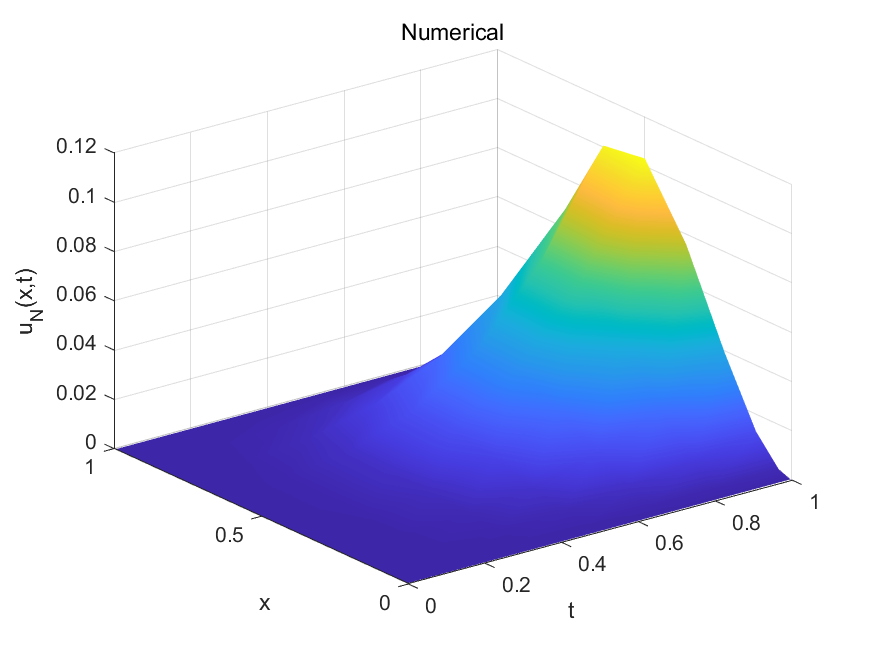}
}
\subfigure[Exact solution] {
\label{fig2:b}
\includegraphics[width=0.3\linewidth]{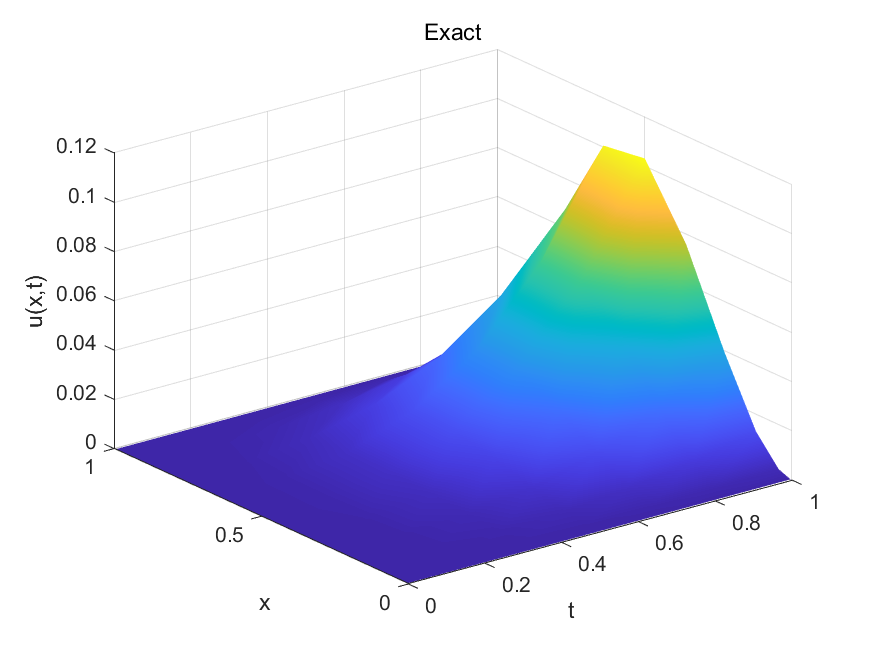}
}
\subfigure[Absolute error] {
\label{fig2:c}
\includegraphics[width=0.3\linewidth]{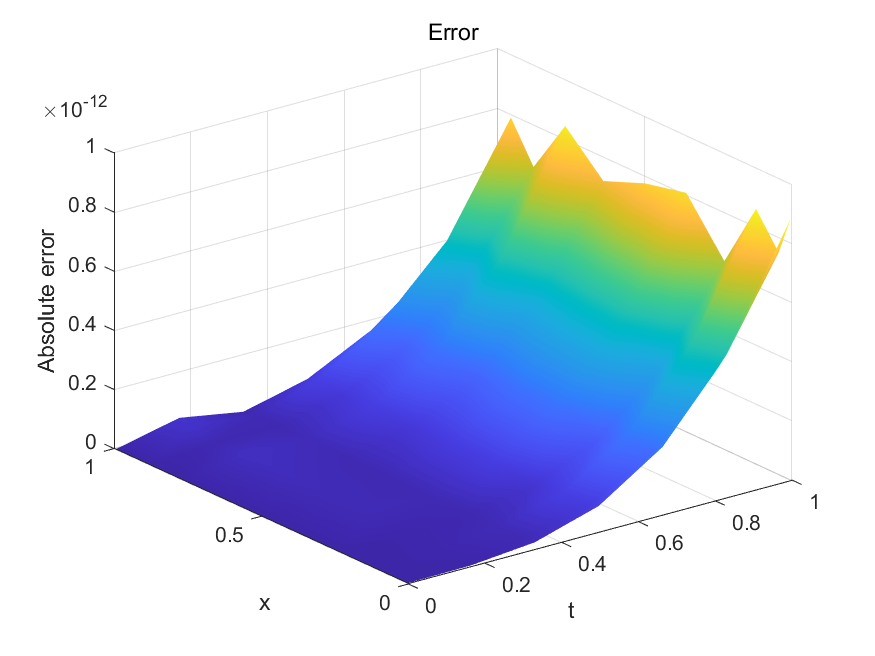}
}
\caption{Three-dimensional plots of the numerical solution (left plot), exact solution (middle plot) and absolute error (right plot) with $M=6,v=m=1,N=10$ and $\alpha=0.5$ (Example 1).}
\label{fig2}
\end{figure}

{\subsection{Example 2}\label{s5:2}}
Here, we consider the time-fractional diffusion problem:
\begin{small}
\begin{align}\label{5.3}
\begin{cases}
D_t^{\alpha} u(x,t) = \partial_{x}^{2} u(x,t)-(x-\frac{\pi}{4})u(x,t)+f(x,t),\ (x,t)\in(\frac{\pi}{4},\frac{3\pi}{4})\times(0,T],\\
u(\frac{\pi}{4},t)-u_x(\frac{\pi}{4},t)=0,\ u(\frac{3\pi}{4},t)+u_x(\frac{3\pi}{4},t)=0,& \\
u_0(x)=-0.05\sin(x).&
\end{cases}
\end{align}
\end{small}
with $c(x)=x-\frac{\pi}{4}\in(0,\frac{\pi}{2})$, whose exact solution is
\begin{align}\label{5.4}
u(x,t)=-0.05E_\alpha(-t^\alpha)\sin x.
\end{align}
We know that the solution (\ref{5.4}) has a weak singularity at $t=0$ for all $\alpha\in(0,1)$ and is less smooth than Example 1 with respect to the time variable. Moreover, it is clear that the vector field $g(t,y)$ in (\ref{4.1}), alike the solution, is also nonsmooth at $t=0$. Subsequently, a routine computation gives rise to
\begin{align*}
f(x,t)=-0.05E_\alpha(-t^\alpha)\left(x-\frac{\pi}{4}\right)\sin x,
\end{align*}
where $E_\alpha(\cdot)$ is the Mittag-Leffler function defined in (\ref{2.3}), which will be computed by using the Matlab code ml \cite{ml2015function,garrappa2015numerical}.

First, we study the convergence behavior of the spatial error of numerical solutions with respect to $N$ in two discrete norms $L^\infty$ and $L^2$, as shown in Fig. \ref{fig3}. The time direction selects a mixed mesh with parameters $M=400$, $v=100$, and $m=1$, with very fine grids chosen to ensure that the error in the time direction can be neglected. Clearly, from Fig. \ref{fig3}, it can be observed that the error decays exponentially, achieving spectral accuracy in two discrete norms.   Next, we test the errors in the temporal direction under two discrete norms: $L^\infty$ and $L^2$. Since the solution and vector field are nonsmooth at $t=0$, a mixed mesh is mandatory. By fixing $N=8,v=15$ and $m=1$, the temporal errors on the mixed mesh with $M=20s\ (s=1,2,3,4,5)$ for four values of $\alpha$ are listed in Table \ref{Tab2}. As can be seen, the temporal errors are dependent on the value of $\alpha$ for the same mesh. This indicates that the smoothness of the true solution at $t=0$ has a certain impact on the error accuracy of the numerical solution. However, fortunately, even when $\alpha$ is small, the achieved error accuracy is still satisfactory. In Fig.\ref{fig4}, we present three-dimensional plots of the numerical solution, exact solution and absolute error for $\alpha=0.5$. From these plots, it is intuitively clear that the numerical solution closely matches the true solution.

\begin{figure}[htbp] \centering
\includegraphics[width=5.5cm]{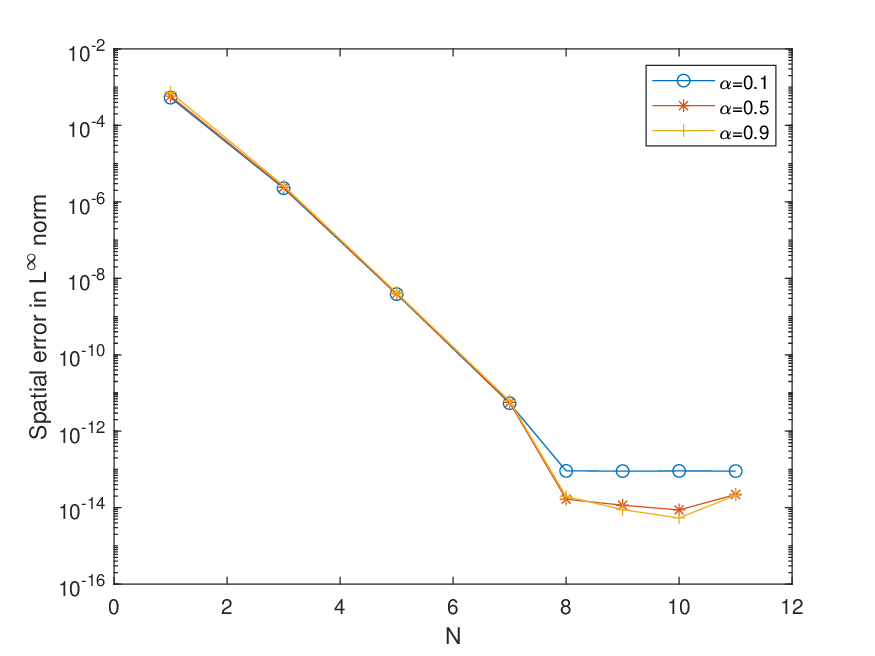}
\includegraphics[width=5.5cm]{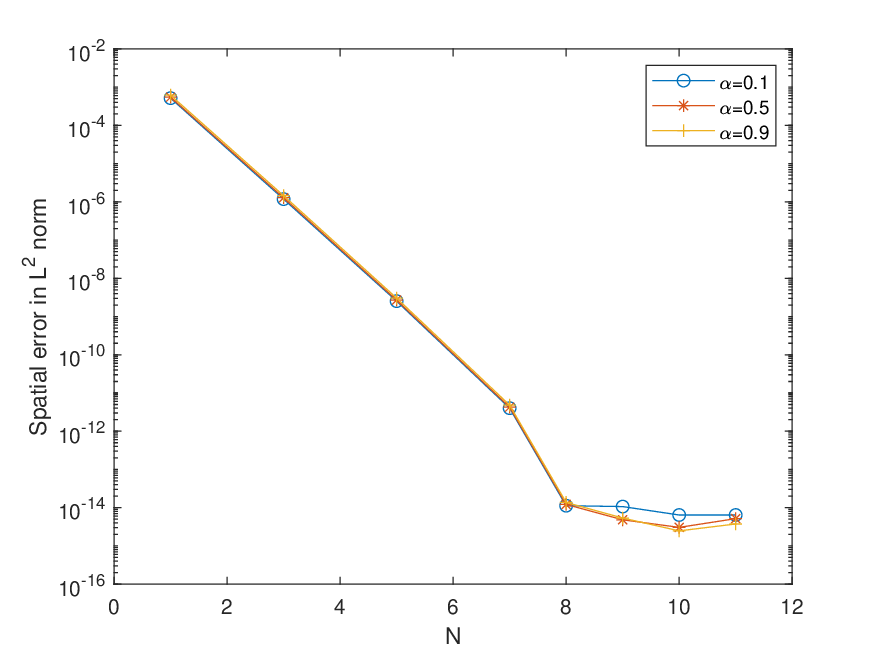}
\caption{The spatial errors in the $L^\infty$ norm (left plot) and $L^2$ norm (right plot) on the mixed mesh with $h_0=1.97\cdot10^{-33},r=2,m=1,h=0.0025$ for different $\alpha$ (Example 2).}
\label{fig3}
\end{figure}

\begin{table}[htbp]
\renewcommand{\arraystretch}{1.3}
\centering
\caption{The temporal errors on the mixed mesh with $v=15,m=1$ and $M=20s\ (s=1,2,3,4,5)$ when $N=8$ (Example 2).}
\scalebox{0.86}{
\begin{tabular}{ccccccccc}
\toprule
\multirow{2}{*}{M} & \multicolumn{2}{c}{$\alpha=0.1$} & \multicolumn{2}{c}{$\alpha=0.3$} & \multicolumn{2}{c}{$\alpha=0.6$} & \multicolumn{2}{c}{$\alpha=0.9$} \\
  & $e_\infty$ & $e_2$ & $e_\infty$ & $e_2$ & $e_\infty$ & $e_2$ & $e_\infty$ & $e_2$ \\
\midrule
20  & 9.688\text{e-}09 & 1.941\text{e-}11 & 1.840\text{e-}10 & 4.541\text{e-}13 & 3.592\text{e-}14 & 1.729\text{e-}14 & 2.385\text{e-}14 & 1.382\text{e-}14 \\
40  & 8.822\text{e-}09 & 1.253\text{e-}11 & 1.227\text{e-}10 & 2.164\text{e-}13 & 2.075\text{e-}14 & 1.262\text{e-}14 & 1.894\text{e-}14 & 1.366\text{e-}14 \\
60  & 8.343\text{e-}09 & 9.690\text{e-}12 & 9.672\text{e-}11 & 1.402\text{e-}13 & 1.853\text{e-}14 & 1.248\text{e-}14 & 2.312\text{e-}14 & 1.369\text{e-}14 \\
80  & 8.017\text{e-}09 & 8.071\text{e-}12 & 8.170\text{e-}11 & 1.031\text{e-}13 & 1.772\text{e-}14 & 1.245\text{e-}14 & 2.936\text{e-}14 & 1.407\text{e-}14 \\
100 & 7.770\text{e-}09 & 7.002\text{e-}12 & 7.167\text{e-}11 & 8.136\text{e-}14 & 1.720\text{e-}14 & 1.245\text{e-}14 & 2.770\text{e-}14 & 1.385\text{e-}14 \\
\bottomrule
\end{tabular}}
\label{Tab2}
\end{table}

\begin{figure}[htbp] \centering
\subfigure[Numerical solution] {
 \label{fig4:a}
\includegraphics[width=0.3\linewidth]{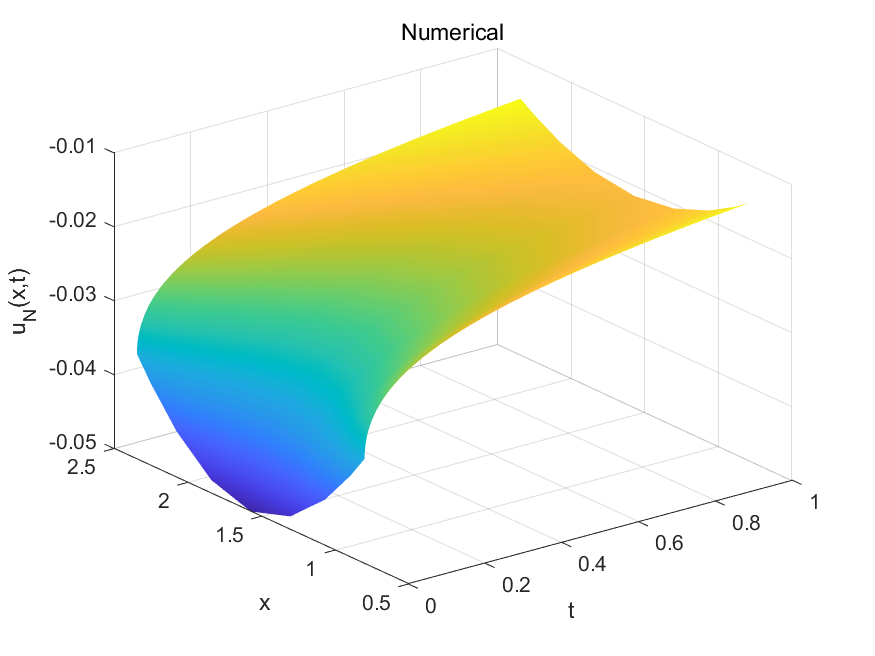}
}
\subfigure[Exact solution] {
\label{fig4:b}
\includegraphics[width=0.3\linewidth]{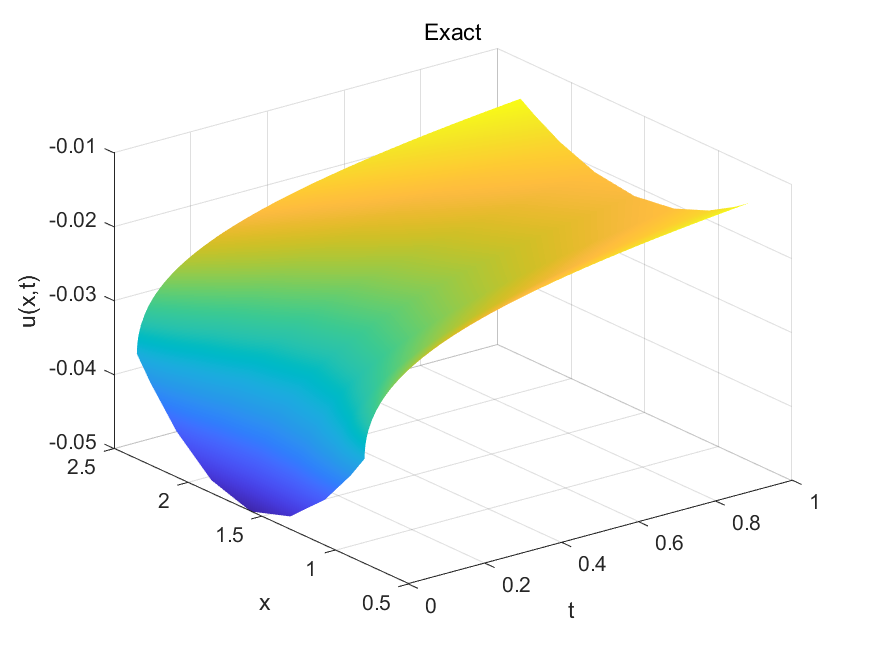}
}
\subfigure[Absolute error] {
\label{fig4:c}
\includegraphics[width=0.3\linewidth]{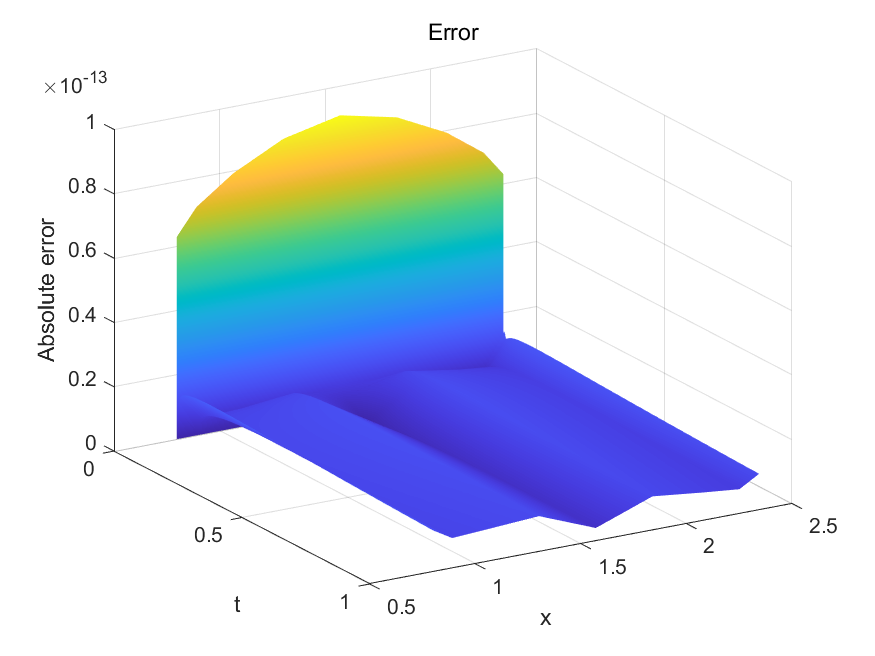}
}
\caption{Three-dimensional plots of the numerical solution (left plot), exact solution (middle plot) and absolute error (right plot) with $M=100,v=15,m=1,N=8$ and $\alpha=0.5$ (Example 2).}
\label{fig4}
\end{figure}
\vskip 0.5cm

{\subsection{Example 3}\label{s5:3}}
Finally, we study a special case of (\ref{1.1})-(\ref{1.3}), namely the time-fractional diffusion problem with homogeneous Dirichlet boundary conditions:
\begin{align}\label{5.5}
\begin{cases}
D_t^{\alpha} u(x,t) = \partial_{x}^{2} u(x,t)+f(x,t),& (x,t)\in(0,1)\times(0,1],\\
u(0,t)=u(1,t)=0,& \\
u_0(x)=0,&
\end{cases}
\end{align}
with an exact solution
\begin{align}\label{5.6}
u(x,t)=t^2\sin(2\pi x),
\end{align}
which is smooth at $t=0$, and the corresponding vector field is also smooth at $t=0$ for any $\alpha\in(0,1)$. In this case, $c(x)=0$ and the function $f(x,t)$ is given by
\begin{align*}
f(x,t)=\sin(2\pi x)\left[\frac{2}{\Gamma(3-\alpha)}t^{2-\alpha}+4\pi^2 t^2\right].
\end{align*}
The problem (\ref{5.5})-(\ref{5.6}) is also addressed in \cite{lin2007finite}, where a finite difference scheme is used for time discretization and the Legendre spectral method is applied in space. In this context, we define the method in \cite{lin2007finite} as \textbf{FDLS}.

To achieve the desired results, we set the time mesh to be the same as in example 1 to ensure a very fine mesh in the temporal direction. This allows us to demonstrate the convergence behavior of numerical errors in the spatial direction, as shown in Fig. \ref{fig5}. As expected, the errors exhibit exponential convergence because the exact solution is smooth in the spatial variables. Next, the errors in the time direction also need be studied. Following the lines of example 1 and fixing $N=11$, we present some results in Table \ref{Tab3}, which suggest that our method can reach the best spatial error accuracy as shown in Fig. \ref{fig5}, even with large step sizes. In Fig. \ref{fig6}, we provide three-dimensional plots of numerical and exact solutions. Overall, the numerical solution obtained by our method is consistent with the true solution. For greater precision, it can be observed from Fig. \ref{fig6:c} that the absolute error reaches $10^{-8}$, indicating that obtained numerical solution can simulate the true solution very well.

Additionally, to highlight the advantages of the our proposed method, we evaluate the performance of our method and the FDLS method in terms of operational efficiency. Since both methods achieve spectral accuracy in space, we focus on the time error under a consistent spatial error. Specifically, the $L^\infty$ norm reaches a level of $10^{-8}$, while the $L^2$ norm reaches a level of $10^{-9}$. Fig. \ref{fig7} plots the time error in the $L^\infty$ norm and $L^2$ norm of the numerical solution versus the CPU time for two different values of $\alpha$. The plots compare our scheme, which uses a uniform time mesh with timestep $h=1/s\ (s=1,2,3,4,5)$, to the FDLS scheme with timestep $h=1/(200s)\ (s=1,2,5,10,20,50)$ for the case $\alpha=0.6$ (shown in Fig. \ref{fig7:a}), and with timestep $h=1/(200s)\ (s=1,5,10,50,100)$ for the case $\alpha=0.99$ (shown in Fig. \ref{fig7:b}). As shown in Fig. \ref{fig7}, we can draw the following clear conclusions: for a given error, the computational cost of our method is lower than that of the FDLS method, and the chosen timestep is also significantly larger than the one selected by the FDLS method.
\vskip 0.6cm

\begin{figure}[htb] \centering
\includegraphics[width=5.5cm]{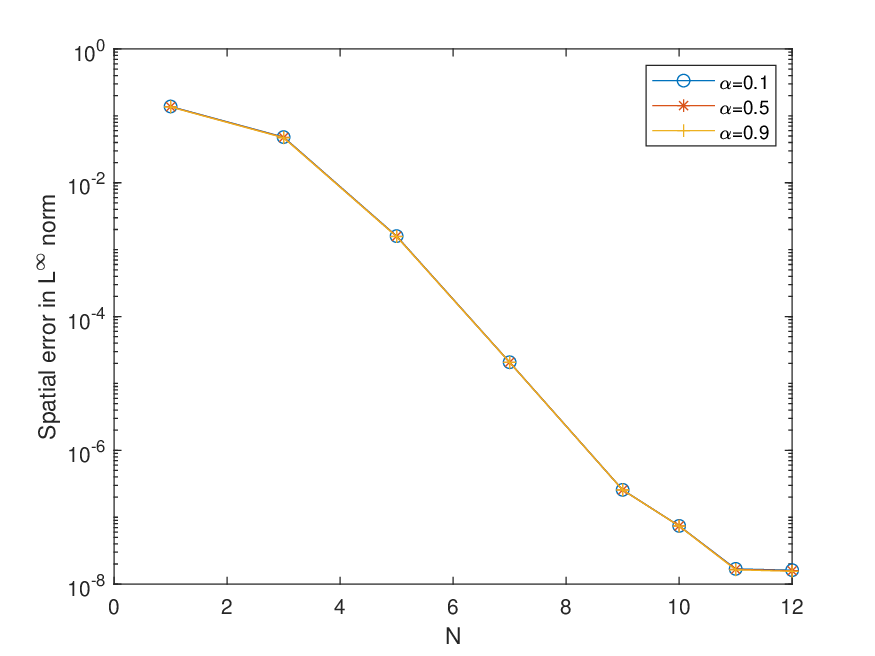}
\includegraphics[width=5.5cm]{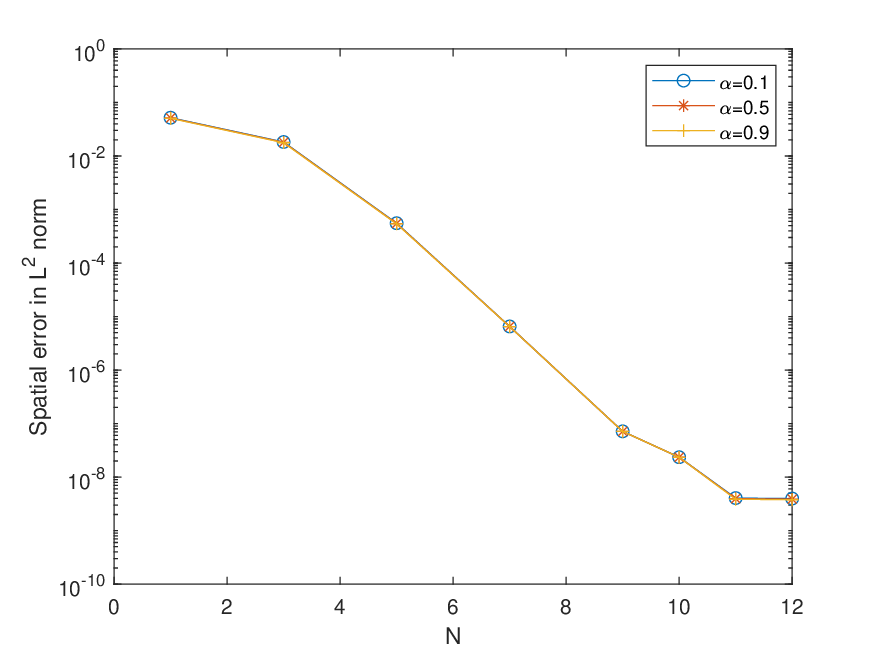}
\caption{The spatial errors in the $L^\infty$ norm (left plot) and $L^2$ norm (right plot) on the mixed mesh with $h_0=4.77\cdot10^{-9},r=2,m=1,h=0.005$ for different $\alpha$ (Example 3).}
\label{fig5}
\end{figure}

\begin{table}[htbp]
\renewcommand{\arraystretch}{1.3}
\centering
\caption{The temporal errors on uniform mesh with $h=1/M$ for four values of $\alpha$ and fixed $N=11$ (Example 3).}
\scalebox{0.86}{
\begin{tabular}{ccccccccc}
\toprule
\multirow{2}{*}{M} & \multicolumn{2}{c}{$\alpha=0.1$} & \multicolumn{2}{c}{$\alpha=0.3$} & \multicolumn{2}{c}{$\alpha=0.6$} & \multicolumn{2}{c}{$\alpha=0.9$} \\
  & $e_\infty$ & $e_2$ & $e_\infty$ & $e_2$ & $e_\infty$ & $e_2$ & $e_\infty$ & $e_2$ \\
\midrule
2 & 1.691\text{e-}08 & 4.804\text{e-}09 & 1.678\text{e-}08 & 4.653\text{e-}09 & 1.658\text{e-}08 & 4.544\text{e-}09 & 1.636\text{e-}08 & 4.606\text{e-}09 \\
3 & 1.691\text{e-}08 & 4.410\text{e-}09 & 1.678\text{e-}08 & 4.352\text{e-}09 & 1.658\text{e-}08 & 4.272\text{e-}09 & 1.636\text{e-}08 & 4.211\text{e-}09 \\
4 & 1.691\text{e-}08 & 4.259\text{e-}09 & 1.678\text{e-}08 & 4.212\text{e-}09 & 1.658\text{e-}08 & 4.138\text{e-}09 & 1.637\text{e-}08 & 4.060\text{e-}09 \\
5 & 1.691\text{e-}08 & 4.186\text{e-}09 & 1.678\text{e-}08 & 4.142\text{e-}09 & 1.658\text{e-}08 & 4.069\text{e-}09 & 1.637\text{e-}08 & 3.987\text{e-}09 \\
6 & 1.691\text{e-}08 & 4.146\text{e-}09 & 1.678\text{e-}08 & 4.103\text{e-}09 & 1.658\text{e-}08 & 4.030\text{e-}09 & 1.637\text{e-}08 & 3.947\text{e-}09 \\
\bottomrule
\end{tabular}}
\label{Tab3}
\end{table}

\begin{figure}[htbp] \centering
\subfigure[Numerical solution] {
 \label{fig6:a}
\includegraphics[width=0.3\linewidth]{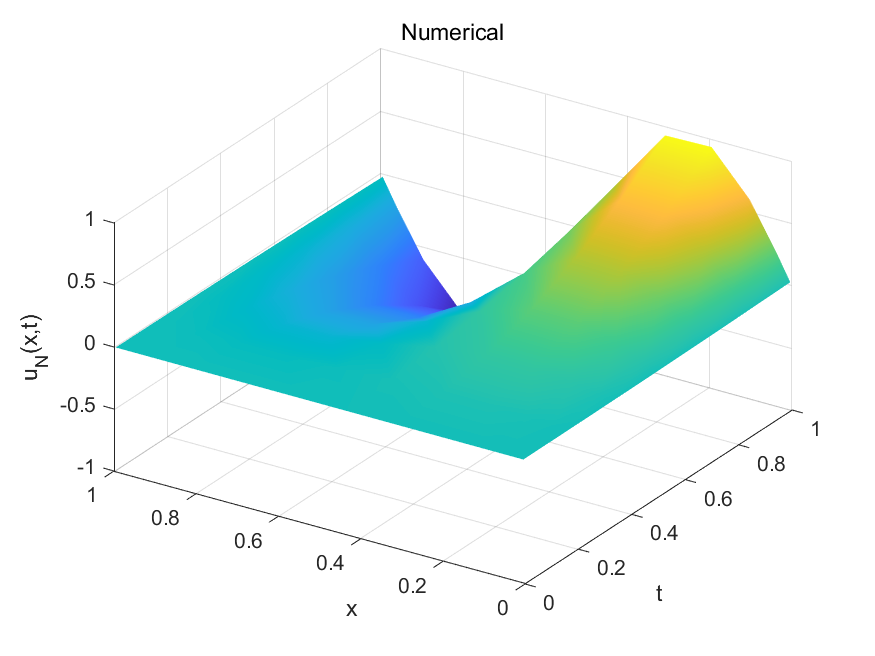}
}
\subfigure[Exact solution] {
\label{fig6:b}
\includegraphics[width=0.3\linewidth]{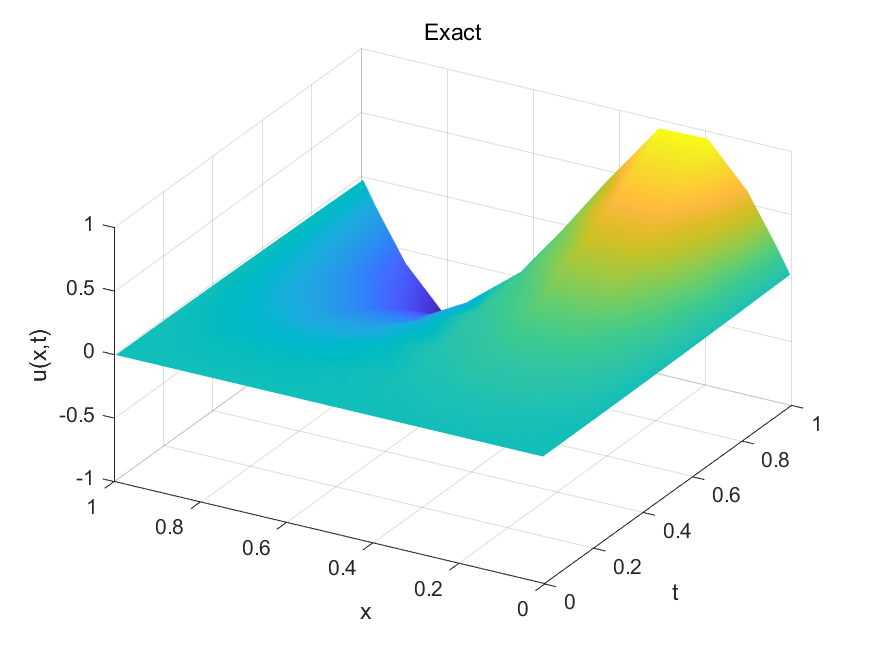}
}
\subfigure[Absolute error] {
\label{fig6:c}
\includegraphics[width=0.3\linewidth]{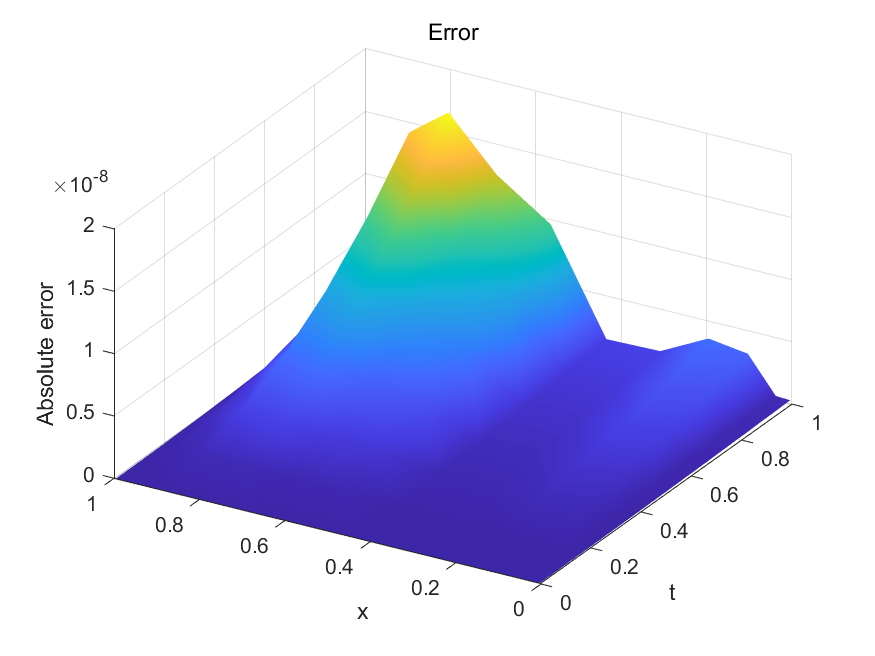}
}
\caption{Three-dimensional plots of the numerical solution (left plot), exact solution (middle plot) and absolute error (right plot) with $M=6,v=m=1,N=11$ and $\alpha=0.5$ (Example 3).}
\label{fig6}
\end{figure}

\begin{figure}[htbp] \centering
\subfigure[$\alpha=0.6$] {
 \label{fig7:a}
\includegraphics[width=5.5cm]{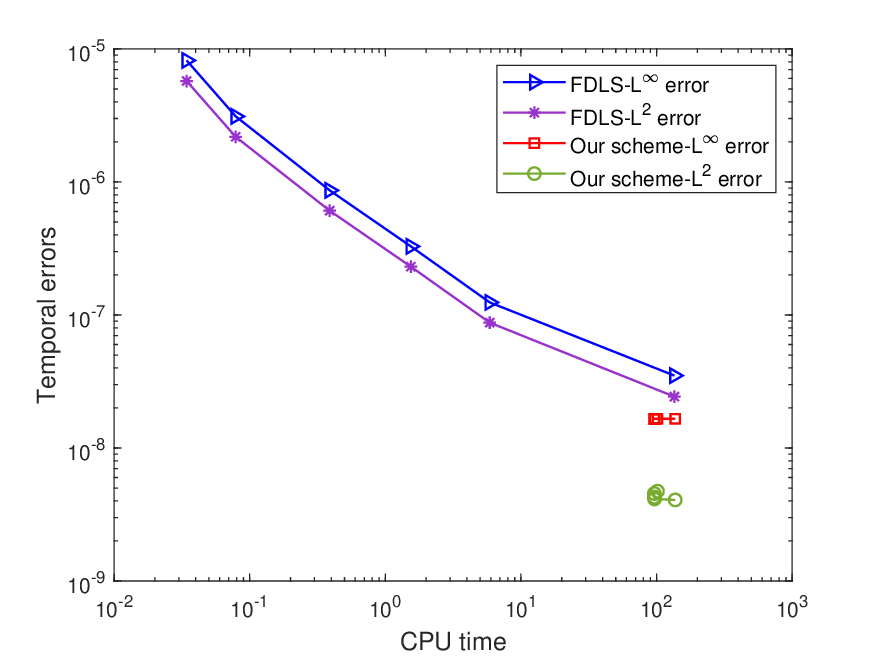}
}
\subfigure[$\alpha=0.99$] {
\label{fig7:b}
\includegraphics[width=5.5cm]{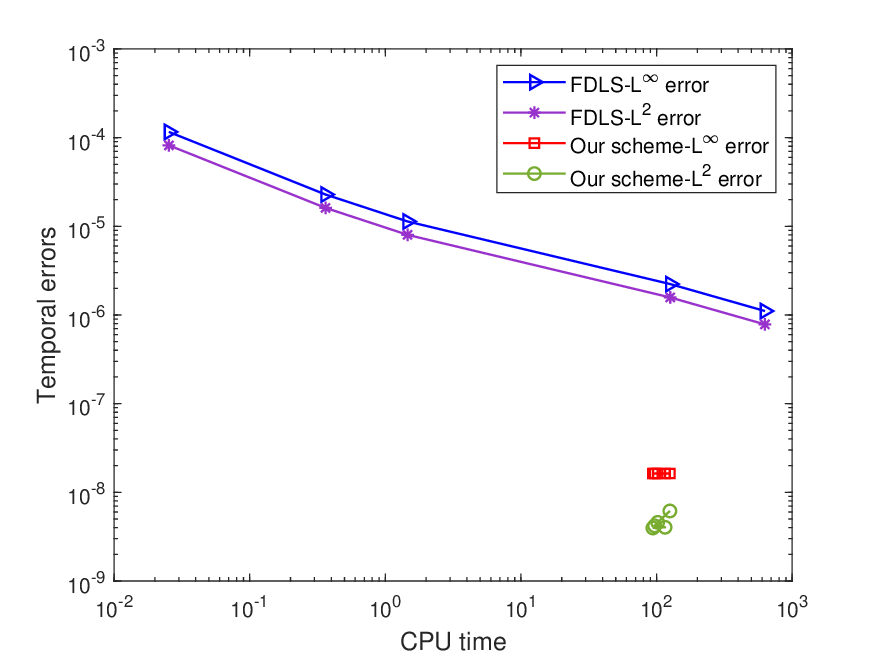}
}
\caption{The temporal errors versus CPU time with spectral accuracy in space over the time interval $t\in[0,1]$ (Example 3).}
\label{fig7}
\end{figure}

{\section{Summary and conclusions}}\label{s6}
\setcounter{subsection}{0} \setcounter{equation}{0}

In this work, we studied a new numerical method for solving time-fractional diffusion equations (\ref{1.1}) subject to Robin boundary conditions (\ref{1.3}), primarily using the line method. Regarding the spatial discretization, we applied the spectral collocation method, where the selected basis functions not only adapt to Robin boundary conditions, but also ensure exponential convergence in space, achieving spectral accuracy. For the semi-discrete FDEs obtained through the spectral collocation scheme,  we used the FHBVM$(k,s)$ method, which is an extension of the HBVM method applied to FDEs and can also be viewed as a class of Runge-Kutta methods, with $k=s=22$ on a mixed mesh. This mixed mesh consists of a graded mesh near the initial time and a uniform mesh thereafter, and is specifically designed to effectively handle problems with a nonsmooth vector field at the origin. Moreover, due to the selection of a sufficiently large $s$, this approach gain a spectral accuracy in time. For the nonlinear problem generated by FHBVMs, to enhance computational efficiency, we employed a mixed iteration strategy to automatically choose between fixed-point iteration and blended iteration under appropriate conditions. Finally, we presented three numerical examples that validate the accuracy and effectiveness of the proposed method and support the theoretical analysis.

It should be mentioned that our method can also handle the problem (\ref{1.1})-(\ref{1.3}) that includes the convection term $q(x)\partial_{x}u$. In future work, we will extend this scheme to a broader class of time-fractional partial differential equations with Robin boundary conditions. Of course, the existence and uniqueness of solutions will be prerequisites.

\vskip 0.5cm
\noindent {\small{\bf Data Availability} ~All data generated or analysed during this study are included in this published article.}
\vskip 0.5cm
\noindent {\bf\large Compliance with Ethical Standards}
\vskip 0.5cm
\noindent {\small{\bf Conflict of interest} ~The authors declare that they have no conflict of interest.}
\vskip 0.8cm


%
%



\end{document}